\documentclass[final,10pt]{article}
\usepackage[utf8]{inputenc}
\usepackage{enumerate,indentfirst}
\usepackage{fontenc,textcomp,latexsym}
\usepackage{amsfonts,amsmath,amsthm,amssymb}
\usepackage{hyperref}
\usepackage{color}
\newtheorem{Assumption}{Assumption}[part]

\newtheorem{Definition}{Definition}[part]

\newtheorem{Lemma}{Lemma}[part]
\newtheorem{Proposition}{Proposition}[part]
\newtheorem{Remark}{Remark}[part]
\newtheorem{Theorem}{Theorem}[part]

\newcommand{\nc}{\newcommand}
\nc{\esssup}{\mathop{\mathrm{ess\,sup}}}
\nc{\essinf}{\mathop{\mathrm{ess\,inf}}}
\nc{\argmax}{\mathop{\mathrm{arg\,max}}}
\def \lb {\label}

\def \P{\mathbb{P}}
\def \N{\mathbb{N}}
\def \R{\mathbb{R}}

\def \E{\mathbb{E}}

\def \1{\mathbf{1}}
\def \A{\mathcal{A}}
\def \Ac{{\cal A}}
\def \Bc{{\cal B}}

\def \Fc{{\cal F}}

\def \Hc{{\cal H}}

\def \Pc{{\cal P}}

\def \Sc{{\cal S}}

\def \nn {\nonumber}
\def \cdlg {\text{c\`adl\`ag}}

\def\enqs{\end{eqnarray*}}
\def\beq{\begin{eqnarray}}
\def\enq{\end{eqnarray}}

\newcommand{\Int} {\displaystyle \int}                       %symbole int?grale
\newcommand{\Sup}     {\displaystyle \sup\limits}

\title{\Large \bf Infinite Horizon Stochastic Impulse Control
with  Delay  and Random Coefficients  }
\author{
{\large Boualem Djehiche}
\thanks{Department of Mathematics, KTH Royal Institute of Technology, 100 44 Stockholm, Sweden. \texttt{e-mail: boualem@math.kth.se}}
{,\,\, \large Said Hamad{\`e}ne}\!\!
\thanks{LMM, Le Mans University, Avenue Olivier Messiaen, 72085 Le Mans, Cedex 9, France. \texttt{e-mail: hamadene@univ-lemans.fr}}
 {,\,\,\large Ibtissem Hdhiri}\thanks{Faculty of Sciences of Gab\`es and LR17ES11, \texttt{e-mail:Ibtissem.Hdhiri@fsg.rnu.tn}}
{\,\,\,\large  and Helmi Zaatra}
\thanks{Faculty of Sciences of Gab\`es and LR17ES11,
\texttt{e-mail:zaatra.helmi@gmail.com}}\,\,. }
\begin{document}
\maketitle
\begin{abstract}
We study a class of infinite horizon impulse control problems with execution delay when the dynamics of the system is described by a general stochastic process adapted to the Brownian filtration. The problem is solved by means of probabilistic tools relying on the notion of Snell envelope and  infinite horizon reflected backward stochastic differential equations. This allows us to establish the existence of an optimal strategy over all admissible strategies.
\end{abstract}

\section{Introduction}

\setcounter{equation}{0} \setcounter{Assumption}{0}
\setcounter{Theorem}{0} \setcounter{Proposition}{0}
\setcounter{Corollary}{0} \setcounter{Lemma}{0}
\setcounter{Definition}{0} \setcounter{Remark}{0} Impulse control is
one of the main topics in the control theory that has attracted a
lot of research activity since it has a wide range of applications
including mathematical finance, insurance, economics, etc. It has been studies since the 70s. For a complete overview of the problem we refer to Bensoussan and Lions (1984).

\vskip0.1cm Several papers are devoted to the Markovian case using
tools from dynamic programming and quasi-variational inequalities,
see e.g. \cite{Harrison,Bensoussan,jeanblanc,Oksendal2,Bruder} among
many others. The first attempt to study the non-Markovian case was
achieved  in Djehiche {\it et al.} \cite{Djehiche} by using
probabilistic tools. Their approach relies on the notion of Snell
envelope  and  reflected backward stochastic differential equations (BSDEs for short) to solve impulse control problems
over a finite time horizon. We also refer to Hdhiri and Karouf
\cite{Hdhiri} for the risk-sensitive case.
\vskip0.1cm

In this work, we study an infinite horizon impulse control with
execution delay, i.e. there is a fixed lag of time $\Delta$ between
the time of decision-making and the time when the execution is
performed. We mention the work by Robin \cite{Robin1} for the
impulse control with delay only in one pending order during the
horizon time. Bayraktar and Egami \cite{Bayraktar} adopt the same
framework of the previous paper for the infinite horizon case, where
they assume the magnitude of the impulse is chosen at the time of
execution. Under restrictive assumptions on the controlled state
process, Bar-Ilan and Sulem \cite{Bar-Ilan1} study an  infinite
horizon impulse control with an arbitrary number of pending orders.
{\O}ksendal and Sulem \cite{Oksendal2} also study the problem with
execution delay when the underlying process is a jump-diffusion.
Hdhiri and Karouf \cite{Hdhiri2} consider a finite horizon impulse
control problem with execution delay where they use  the same
probabilistic tools of \cite{Djehiche}, such as the Snell envelope
notion and  Reflected BSDEs to solve the problem. Due to the delay
$\Delta>0$, when the horizon is finite, this problem turns into the
backward resolution of a finite number of optimal stopping problems
(\cite{Bruder,Palczewski,Hdhiri}).
\medskip

The main contribution of the present work is a solution to an
infinite horizon  impulse control problem with execution delay for a
wide class of stochastic processes adapted to the Brownian
filtration which are not necessarily Markovian. Furthermore, the
running reward functional is not only a deterministic function of
the underlying process but may also be random. Our method relies on
constructing  an approximation scheme for the value function in
terms of a sequence of solutions of infinite horizon reflected
BSDEs. Different from the finite horizon case, the problem now
cannot be reduced to the backward resolution of a finite optimal
stopping problem. The main issue that we solve in this paper is to
establish continuity of the value function of the problem.

The procedure of finding a sequence of optimal stopping times can be
divided into a sequence of steps as follows. Given an initial time
$t$, we find the first time $\tau_{1}$ where it is optimal to
intervene  and we denote the corresponding impulse size
$\beta_{1}^{*}$. Note that this is the first optimal stopping time
after the initial time when the controller may  intervene. The
execution time is not instantaneous, but it occurs after a lag of
time $\Delta$. Next, we proceed to find the first time after
$\tau_{1}+\Delta$ where it is optimal to intervene. This will give
the optimal stopping time $\tau_{2}$ and the corresponding
impulse size $\beta_{2}^{*}$. We continue this procedure over and
over again. \vskip0.1cm

The paper is organized as follows. In section 2, we provide some
preliminaries and recall existence and uniqueness results for
solutions to  infinite horizon reflected BSDEs. In section 3, we
formulate the impulse control problem.  In section 4, we construct
an approximation scheme for the value function of the control
problem, relying on the infinite horizon reflected BSDEs and the
Snell envelope. Section 5, is devoted to establishing  existence of
an optimal impulse control over strategies with a limited number of
impulses. In section 6, we prove the continuity of the value
function and derive an optimal impulse control over all admissible
strategies. Finally, in section 7, we extend the study to  the
risk-sensitive case which involves exponential utilities.  At the end of the paper, in a short appendix, we present the Snell envelope properties
and the notion of predictable and optional projections.

\section{Preliminary results}

\setcounter{equation}{0} \setcounter{Assumption}{0}
\setcounter{Theorem}{0} \setcounter{Proposition}{0}
\setcounter{Corollary}{0} \setcounter{Lemma}{0}
\setcounter{Definition}{0} \setcounter{Remark}{0}

Let $(\Omega,\mathcal{F},\P)$ be a complete
probability space on which is defined a standard $d$-dimensional
Brownian motion $B=(B_{t})_{t\geq 0}$. We denote by
$(\mathcal{F}_{t}^{0}:= \sigma\{B_{s}, s\leq t\})_{t\geq 0}$
the natural filtration of $B$, $(\mathcal{F}_{t})_{t\geq 0}$ its completion
with the $\P$-null sets of $\Fc$ and
$\mathcal{F}_{\infty}=\bigvee_{t \geq 0}\mathcal{F}_{t}$.
Let $\Pc$ be the $\sigma$-algebra on
$\Omega\times [0,\infty[$ of $\mathcal{F}_{t}$-progressively
measurable sets.

For a stochastic process $(y_t)_{t\in [0,\infty)}$ we define its value
at $t=+\infty$ by $y_{\infty}=\limsup_{t\rightarrow \infty}y_t$.
On the other hand, we say that $y$ is continuous at $t=+\infty$ if
$\lim_{t\rightarrow \infty}y_t$ exists. We then set $y_{\infty}=\lim_{t\rightarrow \infty}y_t$. Finally, if $y$ 
is a non-negative (or bounded by below), $\cdlg$, $\Fc_t$-supermartingale then it is continuous at
$t=+\infty$ (\cite{ks}, pp.18).
\medskip

Introduce the following spaces.
\medskip

\noindent i)  $L^{2}=\{\eta:  \mathcal{F}_{\infty}-$ measurable random variable, such that $\E[|\eta|^{^{2}}]<\infty\}$,\\
ii) $\mathcal{H}^{2,m}=\{(v_{t})_{0\leq t <\infty}:$ $\Pc$-measurable, $\R^m$-valued process such that $\E[\int_{0}^{\infty}|v_{s}|^{2}\,ds]<\infty\}$ ($m\ge 1$),\\
iii) $\mathcal{S}^{2}=\{{(y_t)}_{0 \leq t \leq \infty}:$
$\Pc$-measurable process  such that $\E[\sup_{0\leq t\leq \infty}|y_{t}|^{2}]<\infty\}$,\\
iv) $\mathcal{S}_{c}^{2}=\{(y_{t})_{0\leq t\leq \infty}:$ continuous process of  $\mathcal{S}^{2}$\},\\
v) $\mathcal{S}_{i}^{2}=\{(k_{t})_{0\leq t\leq \infty}:$ continuous non-decreasing process of  $\mathcal{S}^{2}$, s.t. $k_{0}=0$\},\\
vi) $\mathcal{T}_{t}=\{\nu, \,\,\mathcal{F}_{t}$-stopping time such that $\mathbb{P}$-a.s. $\nu\geq t\}$.
\medskip

\noindent Next, we give the definition of a solution of an infinite horizon
reflected backward stochastic differential equation with terminal
condition $\xi$, driver $g$ and a lower barrier $X$.

\begin{Definition} We  say that the triple of $\Pc$-measurable processes $(Y_{t},Z_{t},K_{t})_{t \geq 0}$ is a solution of the infinite horizon BSDE associated with $(g,\xi,L)$, if
\begin{equation}
 \label{Eq1}\left\{
    \begin{array}{ll}
      Y\in \Sc^{2}_c, Z\in \Hc^{2,d} \mbox{ and } K\in \Sc_{i}^{2};\\
       \displaystyle Y_{t}=\xi+\int_{t}^{\infty}g(s,Y_{s},Z_{s})\,
ds+K_{\infty}-K_{t}-\int_{t}^{\infty}Z_{s}\,dB_{s},\,\,t\geq 0; \\
      Y_{t}\geq X_{t},\,\, t \geq 0 \,\, \mbox{ and } \,\,\int_0^\infty (Y_t-X_t)dK_t = 0.
    \end{array}
  \right.
\end{equation}
\end{Definition}

We have the following  existence and uniqueness result of the solution of (\ref{Eq1}).

\begin{Theorem} [\cite{
Hamadene1}]\label{Theor}
Assume that
\begin{itemize}
\item[ (i)] $\xi$ is $\Fc_{\infty}$-measurable and belongs to $L^{2}$, the process $X:=(X_{t})_{t\geq 0}$ belongs to  $\mathcal{S}_c^{2}$ and such that $\limsup\limits_{t\rightarrow +\infty} X_{t}\leq \xi $\,\, $\mathbb{P}$-a.s.
\item[(ii)] The driver $g$ is a map from $[0,\infty)\times\Omega\times \R^{1+d}$ to $\R$ which satisfies
\begin{itemize}
\item[(a)] The process $(g(t,0,0))_{t\geq 0}$ belongs to $\mathcal{H}^{2,d}.$
\item[(b)] There exist two positive deterministic borelian functions $u_1$ and $u_2$ from
$\R^+$ into $\R^+$ such that $\int_0^\infty u_1(t) dt < \infty$, $\int_0^\infty u_2^2(t) dt < \infty$ and for every $(y,z)$ and $(y',z')$ in $\R^{1+d}$
$$\P-a.s.,\,\,\,
|g(t,y,z)-g(t,y',z')| \leq u_1(t) |y-y'|+ u_2(t)|z-z'|,\,\,\,\, t\in [0,\infty).
$$
 \end{itemize}
 \end{itemize}
Then there exists a triple of processes $(Y,Z,K)$ which satisfies \eqref{Eq1} and the following representation holds true.
\begin{equation}
\forall t\geq 0,\,\,
Y_{t}=\esssup\limits_{\tau \in
\mathcal{T}_t}\E\left[\int_{t}^{\tau}g(s,Y_{s},Z_{s})\,ds+
X_{\tau} \1_{[\tau <\infty]}+\xi \1_{[\tau =\infty]}|\mathcal{F}_{t}\right]. \label{Yesssup1}
\end{equation}
Furthermore, for any $t\ge 0$, the stopping time
$$
D_t =\left\{\begin{array}{ll}\inf\{s\geq t, \,\,Y_s \leq X_s\} \text{ if finite}, \\
                                          +\infty \quad \hbox{ otherwise,}
                                                       \end{array}
                                                     \right.
                                                     $$
is optimal after $t$ in the sense that
\begin{equation}
Y_{t}=\E \left[\int_{t}^{D_t}g(s,Y_{s},Z_{s})ds+
X_{D_t} \1_{[D_t<\infty]}+\xi \1_{[D_t =\infty]}|\mathcal{F}_{t} \right].\qed\label{optimal}
\end{equation}
\end{Theorem}

\section{Formulation of the impulse problem with delay}
\label{sec3}
\setcounter{equation}{0} \setcounter{Assumption}{0}
\setcounter{Theorem}{0} \setcounter{Proposition}{0}
\setcounter{Corollary}{0} \setcounter{Lemma}{0}
\setcounter{Definition}{0} \setcounter{Remark}{0}
Let  $L=(L_t)_{ t \geq  0}$ be a stochastic process that describes the evolution of
a system which we assume $\Pc$-measurable and with values in $\R^l$. An impulse control  is a sequence of
pairs $\delta = (\tau_n, \xi_n)_{n \geq 1}$ in which $(\tau_n)_{n \geq 1}$ is a
sequence of ${\cal F}_t$-stopping times such that
 $0 \leq \tau_1\leq\ldots\leq \tau_{n} \ldots \;\:\P$-a.s. and $(\xi_n)_{n\geq 1}$
 a sequence of random variables with values in a {\it{ finite}} subset $U:=\{\beta_1,..,\beta_p\}$ of $\R^l$  such that $\xi_n$ is ${\Fc}_{\tau_n}$-measurable. Considering the subset $U$ finite is in line with the fact that, in practice, the controller has only access to limited resources which allows him to exercise
impulses of finite size.

For any $n \geq 1$, the stopping time $\tau_{n}$ stands for the $n$-th time where the controller makes the decision to impulse the system with a magnitude equal to $\xi_{n}$ and  which will be executed after a time lag  $\Delta$. Therefore, we require that  $\tau_{n+1}-\tau_{n}\geq\Delta$, $\P$-a.s., and then we obviously have $\lim_{n\rightarrow +\infty}\tau_{n}= +\infty$.
 \vskip0.1cm

The sequence $\delta = (\tau_n, \xi_n)_{n \geq 1}$ is said to be an admissible strategy of impulse control, and the set of admissible strategies will be denoted by ${\A}$.
\vskip0.1cm
 When the decision maker implements the strategy $\delta=(\tau_{n},\xi_{n})_{n\geq 1}$, the controlled process $L^{\delta}=(L^{\delta}_t)_{t\geq 0})$  is defined as follows. For any $t\ge 0$,
$$L_{t}^{\delta} =\left\{\begin{array}{ll}
                                                         L_{t} \qquad  \text{if}\quad 0 \leq t < \tau_{1}+\Delta,\\
                                                         L_{t}+\xi_{1}+\cdots +\xi_{n} \quad \text{ if} \quad \tau_{n}+\Delta \leq t < \tau_{n+1}+\Delta, n\geq 1,&
                                                       \end{array}
                                                     \right.
 $$
or in a compact form
$$L_{t}^{\delta} =L_{t}+\sum_{n\geq
1}\xi_{n}\1_{[\tau_{n}+\Delta \leq t]}.$$
On the other hand, when the strategy $\delta$ is implemented, the associated total discounted expected payoff (the reward function) is given by:
\begin{equation}
J(\delta):=\E \left[\int_{0}^{\infty} e^{-rs}h(s,L_{s}^{\delta})\;ds-\sum_{n\geq
1}e^{-r(\tau_{n}+\Delta)}\psi(\xi_{n})\right], \label{reward}
\end{equation}
where
\begin{enumerate}[i)]
\item $h$ is a non-negative function which stands for the instantaneous reward and $r$, the discount factor, is a positive real constant.
\item $\psi$ is the cost of making an impulse or intervention  and it has the form
$$\psi(\xi)=k+\phi(\xi),$$
where $k$ (resp. $\phi$) is a positive constant (resp. non-negative function) and stands for the fixed (resp. variable) part of the cost of making an intervention.
\end{enumerate}
The objective is to find an optimal strategy $\delta^{\ast}=(\tau_{n}^{\ast},\xi_{n}^{\ast})_{n\geq1}$, i.e. which satisfies
$$J(\delta^{\ast})=\Sup_{\delta\in \mathcal{A}}J(\delta).$$
\begin{Remark}\normalfont
The process $L$ can take the form
\begin{equation}
L_t=x+\int_0^tb(s,\omega)ds+\int_0^t\sigma(s,\omega)dB_s, \,\,t\geq 0,
\end{equation}
where $b$ (resp. $\sigma$) is a process of $\Hc^{2,1}$ (resp.  $\Hc^{2,d}$). Then $L$ is an It\^o process which is not Markovian and then  the standard methods in e.g. \cite{Bensoussan,jeanblanc,Oksendal2}, etc. based on the Markovian properties do not apply.
\end{Remark}
\vskip0.1cm

Throughout this paper, we make the following assumptions.
\begin{Assumption}\label{assumpt}
i) The functions $h:[0,+\infty)\times\Omega\times\R^{l}\longrightarrow[0,+\infty)$ is
$\Pc\otimes \Bc(\R^l)$-measurable and uniformly bounded by a constant $\gamma$ in all its arguments i.e.,
$$\P\mbox{-a.s.},\,\,\forall\;(t,x)\in[0,+\infty)\times\R^{l},\quad  0\leq h(t,w,x)\leq
\gamma. $$
ii) $\phi$ is a non-negative function defined on $U$. Note that since $U$ is finite, $\phi(\xi)$ is obviously bounded for any $\xi$  random variable with values in $U$.
\end{Assumption}

\section{Iterative scheme}
\label{sec4} In this section, we consider an iterative scheme which relies on infinite horizon reflected BSDEs  in order to find an optimal strategy that maximizes the total discounted expected reward (\ref{reward}).  Let $\nu$ be an $\Fc_t$-stopping time and $\xi$ a finite $\Fc_\nu$-random variable, i.e.,
$\text{card}(\xi(\Omega))<\infty$. Next, let $(Y_{t}^{0}(\nu,\xi),Z_{t}^{0}(\nu,\xi))_{t\geq 0}$ be the solution in $\Sc^2_{c} \times \Hc^{2,d}$ of the following standard BSDE with infinite horizon.
\begin{equation}Y_{t}^{0}(\nu,\xi)=\int_{t}^{\infty}e^{-rs}h(s,L_{s}+\xi)\1_{[s\geq \nu]}\,ds-\int_{t}^{\infty}Z_{s}^{0}(\nu,\xi)\,dB_{s},\quad t\ge 0.
\label{eq 4}
\end{equation}
The solution of \eqref{eq 4} exists and is unique under Assumption \ref{assumpt} thanks to the result by Z.Chen (\cite{Chen}, Theorem $1$). In addition, the process $Y^0(\nu,\xi)$ satisfies, for any $t \geq 0$,
\begin{equation}\label{Y0}
Y_{t}^{0}(\nu,\xi)=\E\left[\int_{t}^{\infty}e^{-rs}h(s,L_{s}+\xi)\1_{[s\geq \nu]}\,ds|\Fc_t\right].
\end{equation}

We will now define  $Y^{n}(\nu,\xi)$ for $n\geq 1$, iteratively in the following way. For any $n \geq1$, let \\$(Y^{n}(\nu,\xi),Z^{n}(\nu,\xi),K^{n}(\nu,\xi))$ be a triple of processes of $\mathcal{S}_{c}^{2}\times \Hc^{2,d} \times \mathcal{S}_{i}^{2}$ which satisfies, for every $t\geq 0$,
\begin{align}
i)  \nonumber \,\,&Y_{t}^{n}(\nu,\xi)=\int_{t}^{\infty}e^{-rs}h(s,L_{s}+\xi)\1_{[s\geq \nu]}\,ds+K_{\infty}^{n}(\nu,\xi)-K_{t}^{n}(\nu,\xi)-\int_{t}^{\infty}Z_{s}^{n}(\nu,\xi)\,dB_{s} \,,
\\
ii)  &\,\, Y_{t}^{n}(\nu,\xi)  \geq  O_{t}^{n}(\nu,\xi):=  \E \left[\Int_{t}^{t+\Delta}e^{-rs}h(s,L_{s}+\xi)\1_{[s\geq \nu]}\,ds |\Fc_{t}\right]
   \nonumber     \\
         &\qquad   \qquad \qquad \qquad +\max\limits_{\beta\in U}\bigg\lbrace \E\bigg\lbrack-e^{-r(t+\Delta)}\psi(\beta)+Y_{t+\Delta}^{n-1}(\nu,\xi+\beta)|\Fc_{t}\bigg\rbrack \bigg\rbrace \,\,, \nonumber\\
        iii)& \,\,\Int_{0}^{\infty}(Y_{t}^{n}(\nu,\xi)-O_{t}^{n}(\nu,\xi))\,dK_{t}^{n}(\nu,\xi)=0.\label{oordn}
\end{align}
Note that once $Y^{n-1}(\nu,\xi)$ is defined, the process $(O_{t}^{n}(\nu,\xi))_{t\ge 0}$ is defined through the optional projections of the non-adapted process $(\int_{t}^{t+\Delta}e^{-rs}h(s,L_{s}+\xi)\1_{[s\geq \nu]}\,ds)_{t\ge 0}$ and \\ $(-e^{-r(t+\Delta)}\psi(\beta)+Y_{t+\Delta}^{n-1}(\nu,\xi+\beta))_{t\ge 0}$ $(\beta \in U)$ (see Part (II) in the appendix for more details). \\

\medskip
We have the following properties of the processes $Y^{n}(\cdot,\cdot)$, ${n\geq 1}$.
\begin{Proposition} \label{Proposition1} For any $n\geq 1$, the triple $(Y^{n}(\nu,\xi),Z^{n}(\nu,\xi),K^{n}(\nu,\xi))$ is well-posed and satisfies, for all $ t\ge 0$,
 \begin{equation}
 Y_{t}^{n}(\nu,\xi)=\esssup\limits_{\tau \in
\mathcal{T}_t}\E\left[\int_{t}^{\tau} e^{-rs}h(s,L_{s}+\xi)\1_{[s\geq \nu]}\,ds+  O_{\tau}^{n}(\nu,\xi)|\mathcal{F}_{t}\right].
\label{ordrenn}\end{equation}
Moreover, we have \\
i) for all $t\geq 0$
\begin{equation} 0 \leq Y_{t}^{n}(\nu,\xi)\leq \frac{\gamma}{r} e^{-rt}.
\label{bornee}\end{equation}
ii) For all $n\ge 0$ and $t\ge 0$,
\begin{equation}\label{compyn} Y_{t}^{n}(\nu,\xi)\leq Y_{t}^{n+1}(\nu,\xi).\end{equation}
\end{Proposition}
\begin{proof} We will proceed by induction. Let $\nu$ be a stopping time, $\xi$ a generic $\mathcal{F}_{\nu}$-measurable random variable. As previously noted,  for $n=0$, the pair $(Y_{t}^{0}(\nu,\xi),Z_{t}^{0}(\nu,\xi))_{t\geq 0}$ exists, belongs to $\Sc^2_{c} \times \Hc^{2,d}$ and satisfies (\ref{bornee}) since $0\leq h\leq \gamma$.

Consider now the case $n=1$. First note that the process
$O^{1}(\nu,\xi)$ belongs to $\Sc^2_{c}$ (by Appendix, Part (II)) and
$\underset{t\to \infty}{\lim}O^{1}_{t}(\nu,\xi)=0$. Actually this
holds true since $Y^{0}(\nu,\xi)$ is continuous and
$\underset{t\to \infty}{\lim}Y_{t}^{0}(\nu,\xi)=0$ by \eqref{bornee}.
Therefore the triple of processes
$(Y^{1}(\nu,\xi),Z^{1}(\nu,\xi),K^{1}(\nu,\xi))$ is well defined
through the BSDE \eqref{oordn} and by (\ref{Yesssup1}) satisfies
\eqref{ordrenn}. Finally, for $t \geq 0$,
\begin{eqnarray}\nonumber
O_{t}^{1}(\nu,\xi)&=&\E \bigg\lbrack \Int_{t}^{t+\Delta}  e^{-rs}  h(s,L_{s}+\xi)\1_{[s\geq \nu]}\,ds|\mathcal{F}_{t}\bigg\rbrack \\ \nonumber
&+&\max\limits_{\beta\in U}\{\E\bigg\lbrack e^{-r(t+\Delta)}(-\psi(\beta))+Y_{t+\Delta}^{0}(\nu,\xi+\beta)|\mathcal{F}_{t}\bigg\rbrack \}\\\nonumber
&\leq& \E\bigg\lbrack \frac{\gamma}{r}\{  e^{-rt}- e^{-r(t+\Delta)}\}-ke^{-r(t+\Delta)}+\frac{\gamma}{r}e^{-r(t+\Delta)}|\mathcal{F}_{t}\bigg\rbrack \\
&\leq& \frac{\gamma}{r}  e^{-rt}.\label{48}
\end{eqnarray}
Again, by  the characterization (\ref{ordrenn}), we have, for every $t\ge 0$,
\begin{eqnarray}
Y_{t}^{1}(\nu,\xi)=\esssup\limits_{\tau \in
\mathcal{T}_t}\E \left[\int_{t}^{\tau} e^{-rs}h(s,L_{s}+\xi)\1_{[s\geq \nu]}\,ds+  O_{\tau}^{1}(\nu,\xi)|\mathcal{F}_{t}\right].
\end{eqnarray}
Therefore,
\begin{eqnarray*}
0\le Y_{t}^{1}(\nu,\xi) & \leq &\esssup\limits_{\tau \in
\mathcal{T}_t}\E\bigg\lbrack \int_{t}^{\tau}\gamma e^{-rs}ds+\frac{\gamma}{r}  e^{-r\tau} |\mathcal{F}_{t}\bigg\rbrack \\
& \leq & \esssup\limits_{\tau \in
\mathcal{T}_t}\E\bigg\lbrack \frac{\gamma}{r} (e^{-rt}-e^{-r\tau})+\frac{\gamma}{r}  e^{-r\tau} |\mathcal{F}_{t}\bigg\rbrack=\frac{\gamma}{r} e^{-rt}.
\end{eqnarray*}
Let us now assume that for some $n$ the triple
$(Y^{n}(\nu,\xi),Z^{n}(\nu,\xi),K^{n}(\nu,\xi))$, for any $\xi\in \Fc_\nu$,  is well-posed and that
\eqref{ordrenn}-\eqref{bornee} hold true. The process
$O^{n+1}(\nu,\xi)$ belongs to $\Sc^2_{c}$ as the predictable
projection of a continuous process and $\underset{t\to \infty}{\lim}O^{n+1}(\nu,\xi)=0$ by \eqref{bornee}  which is valid by the induction hypothesis. Therefore the triple
$(Y^{n+1}(\nu,\xi),Z^{n+1}(\nu,\xi),K^{n+1}(\nu,\xi))$ is well-posed
by the BSDE \eqref{oordn} and by (\ref{Yesssup1}) satisfies
\eqref{ordrenn}. Finally, the fact that $Y^{n+1}(\nu,\xi)$ satisfies
\eqref{bornee} can be obtained as for $Y^{1}(\nu,\xi)$ since
$O^{n+1}(\nu,\xi)$ satisfies \eqref{48}. The induction is now
complete.

Finally we have also \eqref{compyn} by comparison of solutions of
reflected BSDEs since we obviously have, for any $\in \Fc_\nu$,
$Y^{0}(\nu,\xi)\le Y^{1}(\nu,\xi)$ and we conclude by using an
induction argument.
\end{proof}
\begin{Remark} \normalfont
 Since $\text{card}(\xi(\Omega))$ is finite, then $\xi$ takes only a finite number of values $k_{1},\ldots,k_{m}$. Therefore,  
using the uniqueness of the solution of the BSDE \eqref{oordn} it follows immediately  that,  for any $t\geq \nu$,
\begin{equation}\label{determinynn}
Y_{t}^{n}(\nu,\xi)=\sum_{k=1}^{m} Y_{t}^{n}(\nu,k_{i})\1_{\{\xi=k_{i}\}}.
\end{equation}
This  means that $Y_{t}^{n}(\nu,\xi)$ is determined by $Y_{t}^{n}(\nu,\theta)$,
for $\theta$ constant which belongs to $\xi(\Omega)$.
On the other hand take the limit w.r.t. $n$ to obtain
\begin{equation}\label{determiny}
Y_{t}(\nu,\xi)=\sum_{k=1}^{m} Y_{t}(\nu,k_{i})\1_{\{\xi=k_{i}\}}.
\end{equation}
\end{Remark}
\begin{Proposition} \label{Proposition2} \, Let $\nu$ be a stopping time and $\xi$  an $\mathcal{F}_{\nu}$-measurable random  variable, then
\begin{enumerate}[i)]
\item the sequence $(Y^{n}(\nu,\xi))_{n\geq 0}$ converges increasingly and pointwisely $\P$-a.s. to a c\`adl\`ag process
$Y(\nu,\xi)$ which satisfies, for all $t\geq 0$,
\begin{equation}
Y_{t}(\nu,\xi)=\esssup\limits_{\tau \in
\mathcal{T}_t}\E\left[\int_{t}^{\tau}e^{-rs}h(s,L_{s}+\xi)\1_{[s\geq \nu]}\,ds+ O_{\tau}(\nu,\xi)|\mathcal{F}_{t}\right],
\label{limit}\end{equation}
where
\begin{eqnarray}\begin{array}{lll}
O_{t}(\nu,\xi)=\E\bigg\lbrack \Int_{t}^{t+\Delta}e^{-rs}h(s,L_{s}+\xi)\1_{[s\geq \nu]}\,ds|\mathcal{F}_{t}\bigg\rbrack
         \\ \qquad\qquad\qquad\qquad\qquad+\max\limits_{\beta\in U}\bigg\lbrace \E\bigg\lbrack -e^{-r(t+\Delta)}\psi(\beta)+Y_{t+\Delta}(\nu,\xi+\beta)|\mathcal{F}_{t}\bigg\rbrack \bigg\rbrace.\label{limito}
         \end{array}
         \end{eqnarray}
 \item If  $\nu'$ is a stopping time  satisfying $\nu\leq\nu'$, then $\P$.a.s., $Y_{t}(\nu,\xi)=Y_{t}(\nu',\xi)$ for all $t\geq\nu'$.
\end{enumerate}
\end{Proposition}

\begin{proof}
i) From Proposition \ref{Proposition1}, we have that the sequence  $(Y_{t}^{n}(\nu,\xi))_{n\geq 0}$ is increasing and satisfies, for any $n\geq 0$,
   $$
   0\le Y_{t}^{n}(\nu,\xi) \leq \frac{\gamma}{r}e^{-rt}.
   $$
  Then taking the limit as $n\rightarrow \infty$, we obtain that the sequence $(Y^{n}(\nu,\xi))_{n\geq0}$ converges to the $\Pc$-measurable process $Y(\nu,\xi)$ satisfying
 \begin{equation} \label{BorneY}
  \forall t\ge 0,\,\quad    0 \leq Y_{t}(\nu,\xi)\leq\frac{\gamma}{r}e^{-rt}.
\end{equation}

Let us now  show that $(Y_{t}(\nu,\xi))_{t\ge 0}$ is c\`adl\`ag.
Indeed, by (\ref{ordrenn}) it follows that the process $\left(
Y_{t}^{n}(\nu,\xi)+ \int_{0}^{t}e^{-rs}h(s,L_{s}+\xi)\1_{[s\geq
\nu]}ds\right)_{t\geq 0}$ is a continuous supermartingale which
converges increasingly and pointwisely to the process $\left(
Y_{t}(\nu,\xi)+ \int_{0}^{t}e^{-rs}h(s,L_{s}+\xi)\1_{[s\geq
\nu]}ds\right)_{t\geq 0}$, which is c\`adl\`ag, as a limit of
increasing sequence of continuous supermatingales (for further
details, see Dellacherie and Meyer Vol. B, pp. 86). In particular
$\left( Y_{t}(\nu,\xi) \right)_{t\geq 0}$ is c\`adl\`ag. Therefore
the  process $(O_{t}(\nu,\xi))_{t\ge 0}$ is also c\`adl\`ag (see
Part (II) in Appendix). To complete the proof, it is enough to use
point v) of Part (I) in Appendix and \eqref{ordrenn} since $(O^{n}(\nu,\xi))_{n\ge 1}\nearrow
O(\nu,\xi)$ pointwisley.

ii) We proceed by induction on $n$. Since the solution of the BSDE
  $$Y_{t}^{0}(\nu,\xi)=\int_{t}^{\infty} e^{-rs} h(s,L_{s}+\xi)\1_{[s\geq \nu]}\,ds-\int_{t}^{\infty}Z_{s}^0(\nu,\xi)\,dB_{s},\quad t\ge 0,$$
is unique, it follows that, for any $\xi\in\mathcal{F}_{\nu}$, $Y_{t}^{0}(\nu,\xi)=Y_{t}^{0}(\nu',\xi)$ for any $t\geq\nu'$. Suppose now that the property is also valid for some $n$, i.e. for every $\xi \in \mathcal{F}_{\nu}$,
$Y_{t}^{n}(\nu,\xi)=Y_{t}^{n}(\nu',\xi)$ for any $t\geq\nu'$. Then $O_{t}^{n+1}(\nu,\xi)=O_{t}^{n+1}(\nu',\xi), \,\,\\forall t\geq\nu'$. Also by uniqueness of the solution of (\ref{oordn}), we have the following equality.
 \begin{equation}\label{remplace}
 \forall \xi \in \mathcal{F}_{\nu},\,\quad Y_{t}^{n+1}(\nu,\xi)=Y_{t}^{n+1}(\nu',\xi),\,\,\forall t\geq\nu'.
  \end{equation}
Hence, the property holds true for any $n\geq 0$, therefore by taking the limit as $n\to + \infty$, we obtain  the proof of the claim.
\end{proof}
\section{Infinite delayed  impulse control with a finite number of interventions}
In this section we consider the case when the controller is allowed to make use of a finite number $n\ge 1$ at most of interventions. Let us define the set of bounded (by $n$) strategies by
$$\mathcal{A}_{n}:=\{(\tau_{k},\xi_{k})_{k\geq 0} \in \mathcal{A},\, \text{such that}\quad \tau_{n}+\Delta = + \infty,\, \mathbb{P}-a.s \}.$$
$\mathcal{A}_{n}$ is the set of strategies where only $n$ impulses
at most are made. We state now the main result of this section.
\begin{Proposition}\lb{53}
Let $n\ge 1$ be fixed. Then there exists a strategy $\delta_{n}^{*}$ which belongs to $\mathcal{A}_{n}$ such that
$$ Y_{0}^{n}(0,0)= \Sup_{\delta \in \mathcal{A}_{n} }J(\delta)=J(\delta^{*}_{n})$$
which means that $\delta_{n}^{*}$ is optimal in $\mathcal{A}_{n}$.
\end{Proposition}
\begin{proof} We first define the strategy $\delta_{n}^{*}$. Let $\tau_{0}^{n}$ be the stopping time defined as
$$\tau_{0}^{n}= \begin{cases}\inf\{ s\in[0,\infty),\,O_{s}^{n}(0,0)\geq Y_{s}^{n}(0,0)\},\\ +\infty \quad \text{otherwise}.\end{cases}$$
Then
\begin{eqnarray*}
O_{\tau_{0}^{n}}^{n}(0,0)&:=&\E \bigg\lbrack
\Int_{\tau_{0}^{n}}^{\tau_{0}^{n}+\Delta}  e^{-rs}  h(s,L_{s})ds|
\mathcal{F}_{\tau_{0}^{n}}
\bigg\rbrack\\ &+&
\max\limits_{\beta\in U}\bigg\lbrace\E \bigg\lbrack
e^{-r(\tau_{0}^{n}+\Delta)}(-\psi(\beta))+ Y_{\tau_{0}^{n}+\Delta}^{n-1}(0,\beta)|
\mathcal{F}_{\tau_{0}^{n}}]\bigg\rbrace
\\ &=& \E \bigg\lbrack \Int_{\tau_{0}^{n}}^{\tau_{0}^{n}
+\Delta} e^{-rs}h(s,L_{s})ds|\mathcal{F}_{\tau_{0}^{n}}\bigg\rbrack \\ &+&
\max\limits_{\beta\in U}\bigg\lbrace \E \bigg\lbrack e^{-r(\tau_{0}^{n}+\Delta)}
(-\psi(\beta))+Y_{\tau_{0}^{n}+\Delta}^{n-1}(\tau_{0}^{n},\beta)|\mathcal{F}_{\tau_{0}^{n}}
\bigg\rbrack
\end{eqnarray*}
since, as mentioned previously in \eqref{remplace},
$Y_{\tau_{0}^{n}+\Delta}^{n-1}(\tau_{0}^{n},\beta)=Y_{\tau_{0}^{n}+\Delta}^{n-1}(0,\beta)$
for any $\beta \in U$. Therefore, as $U$ is finite, there exists
$\beta_{0}^{n}$ with values in $U$,
$\mathcal{F}_{\tau_{0}^{n}}$-measurable such that
\begin{equation}\lb{conformite} O_{\tau_{0}^{n}}^{n}(0,0)=\E \bigg\lbrack \Int_{\tau_{0}^{n}}^{\tau_{0}^{n}+\Delta}  e^{-rs}  h(s,L_{s})ds - e^{-r(\tau_{0}^{n}+\Delta)}\psi(\beta_{0}^{n})+Y_{\tau_{0}^{n}+\Delta}^{n-1}(\tau_{0}^{n},\beta_{0}^{n})|\mathcal{F}_{\tau_{0}^{n}}\bigg\rbrack.\end{equation}

\noindent The r.v. $\beta_{0}^{n}$ can be constructed in the
following way. For $i=1,\ldots,p$, let $\A_i$ be the set,
$$
 \A_i:=\bigg\lbrace
   \max\limits_{\beta\in U}\E \bigg\lbrack e^{-r(\tau_{0}^{n}+\Delta)}
   (-\psi(\beta))+Y_{\tau_{0}^{n}+\Delta}^{n-1}(\tau_{0}^{n},\beta)
   |\mathcal{F}_{\tau_{0}^{n}}\bigg\rbrack =
   \E \bigg\lbrack e^{-r(\tau_{0}^{n}+\Delta)}
   (-\psi(\beta_i))+Y_{\tau_{0}^{n}+\Delta}^{n-1}(\tau_{0}^{n},\beta_i)
   |\mathcal{F}_{\tau_{0}^{n}}\bigg\rbrack   \bigg\rbrace.
   $$
We then define $\beta_{0}^{n}$ as
$$
\beta_{0}^{n}=\beta_1 \mbox{ on } \A_1 \text{ and  }
\beta_{0}^{n}=\beta_j \mbox{ on
}\A_j\backslash\bigcup_{k=1}^{j-1}\A_k \text{ for }j=2,\ldots,p.
$$
Therefore, by \eqref{determinynn}, $\beta^n_0$ satisfies
\eqref{conformite}. Indeed,
\begin{align}\label{inside1}
\nn&\E\bigg\lbrack - e^{-r(\tau_{0}^{n}+\Delta)}\psi(\beta_{0}^{n})+Y_{\tau_{0}^{n}+\Delta}^{n-1}(\tau_{0}^{n},\beta_{0}^{n})|\mathcal{F}_{\tau_{0}^{n}}\bigg\rbrack\\\nn&\nn\qquad=
\E\bigg\lbrack \sum_{i=1,p}1_{\A_i}\{- e^{-r(\tau_{0}^{n}+\Delta)}\psi(\beta_i)+Y_{\tau_{0}^{n}+\Delta}^{n-1}(\tau_{0}^{n},\beta_i)\}|\mathcal{F}_{\tau_{0}^{n}}\bigg\rbrack\nn\\&\nn\qquad
=
\sum_{i=1,p}1_{\A_i}\E\bigg\lbrack \{- e^{-r(\tau_{0}^{n}+\Delta)}\psi(\beta_i)+Y_{\tau_{0}^{n}+\Delta}^{n-1}(\tau_{0}^{n},\beta_i)\}|\mathcal{F}_{\tau_{0}^{n}}\bigg\rbrack\\&\nn\qquad=
\sum_{i=1,p}1_{\A_i}\max\limits_{\beta\in U}\E \bigg\lbrack e^{-r(\tau_{0}^{n}+\Delta)}
   (-\psi(\beta))+Y_{\tau_{0}^{n}+\Delta}^{n-1}(\tau_{0}^{n},\beta)
   |\mathcal{F}_{\tau_{0}^{n}}\bigg\rbrack\\&\qquad=
   \max\limits_{\beta\in U}\E \bigg\lbrack e^{-r(\tau_{0}^{n}+\Delta)}
   (-\psi(\beta))+Y_{\tau_{0}^{n}+\Delta}^{n-1}(\tau_{0}^{n},\beta)
   |\mathcal{F}_{\tau_{0}^{n}}\bigg\rbrack
\end{align}
which yields the claim.
\medskip

\noindent Next, for any $k \in\{ 1,\ldots,n-1 \}$, once $( \tau_{k-1}^{n},
\beta_{k-1}^{n})$ is defined, we define $\tau_{k}^{n}$ by
$$ \tau_{k}^{n}=\inf \bigg\lbrace s\geq \tau_{k-1}^{n}+\Delta , O_{s}^{n-k}(\tau_{k-1}^{n},\beta_{0}^{n}+\cdots +\beta_{k-1}^{n} )\geq Y_{s}^{n-k}(\tau_{k-1}^{n},\beta_{0}^{n}+\cdots +\beta_{k-1}^{n} )\bigg\rbrace.$$
and $\beta_{k}^{n}$ an $\mathcal{F}_{\tau_{k}^{n}}$-r.v. valued in
$U$ such that
\begin{eqnarray*}
O_{\tau_{k}^{n}}^{n-k}(\tau_{k-1}^{n},\beta_{0}^{n}+\cdots
+\beta_{k-1}^{n} )&=&\E \bigg\lbrack
\Int_{\tau_{k}^{n}}^{\tau_{k}^{n}+\Delta}  e^{-rs}
h(s,L_{s}+\beta_{0}^{n}+\cdots +\beta_{k-1}^{n} )ds\\ &-&
e^{-r(\tau_{k}^{n}+\Delta)}\psi(\beta_{k}^{n})+Y_{\tau_{k}^{n}+\Delta}^{n-k-1}(\tau_{k}^{n},\beta_{0}^{n}+\cdots
+\beta_{k-1}^{n}+\beta_{k}^{n} )|\mathcal{F}_{\tau_{k}^{n}}
\bigg\rbrack
\end{eqnarray*}
where we have used the equality
$Y_{\tau_{k}^{n}+\Delta}^{n-k-1}(\tau_{k-1}^{n},\beta^n_0+\cdots
+\beta_{k-1}^{n}+\beta)=Y_{\tau_{k}^{n}+\Delta}^{n-k-1}(\tau_{k}^{n},\beta_{0}^{n}+\cdots
+\beta_{k-1}^{n}+\beta)$ for any $\beta \in U$(see
\eqref{remplace}).
\medskip

We now show that $\delta_n^*$ is optimal. First note that from the characterisation(\ref{ordrenn}), we have that
$$ Y_{0}^{n}(0,0)=\sup\limits_{\tau \in
\mathcal{T}_0}\E \bigg\lbrack \int_{0}^{\tau} e^{-rs}h(s,L_{s})\,ds+  O_{\tau}^{n}(0,0)\bigg\rbrack.$$
Moreover, since the process $ O^{n}(0,0)$ is continuous on $[0,\infty]$ ($O_\infty^{n}(0,0)=\underset{t\to \infty}{\lim} O^{n}_t(0,0)=0$), then the stopping time $\tau_{0}^{n}$ is optimal after $0$. It follows that
\begin{equation}\label{Yn}
Y_{0}^{n}(0,0)= \E \bigg\lbrack \int_{0}^{\tau_{0}^{n}} e^{-rs}h(s,L_{s})\,ds+  O_{\tau_{0}^{n}}^{n}(0,0)\bigg\rbrack.
\end{equation}
But,
\begin{eqnarray*}\begin{array}{lll}
O_{\tau_{0}^{n}}^{n}(0,0):=\E \bigg\lbrack \Int_{\tau_{0}^{n}}^{\tau_{0}^{n}+\Delta}  e^{-rs}  h(s,L_{s}) ds|\mathcal{F}_{\tau_{0}^{n}} \bigg\rbrack \\ \qquad\qquad\qquad+ \max\limits_{\beta\in U}\bigg\lbrace \E \bigg\lbrack e^{-r(\tau_{0}^{n}+\Delta)}(-\psi(\beta))+Y_{\tau_{0}^{n}+\Delta}^{n-1}(0,\beta)|\mathcal{F}_{\tau_{0}^{n}}\bigg\rbrack \bigg\rbrace
\\  \qquad\qquad = \E \bigg\lbrack \Int_{\tau_{0}^{n}}^{\tau_{0}^{n}+\Delta}  e^{-rs}  h(s,L_{s})ds|\mathcal{F}_{\tau_{0}^{n}}\bigg\rbrack \\ \qquad\qquad \qquad+  \max\limits_{\beta\in U} \bigg\lbrace \E \bigg\lbrack e^{-r(\tau_{0}^{n}+\Delta)}(-\psi(\beta))+Y_{\tau_{0}^{n}+\Delta}^{n-1}(\tau_{0}^{n},\beta)|\mathcal{F}_{\tau_{0}^{n}} \bigg\rbrack  \bigg\rbrace
\\  \qquad\qquad =\E \bigg\lbrack \Int_{\tau_{0}^{n}}^{\tau_{0}^{n}+\Delta}  e^{-rs}  h(s,L_{s})ds- e^{-r(\tau_{0}^{n}+\Delta)} \psi(\beta_{0}^{n})+Y_{\tau_{0}^{n}+\Delta}^{n-1}(\tau_{0}^{n},\beta_{0}^{n})|\mathcal{F}_{\tau_{0}^{n}} \bigg\rbrack.
\end{array}
\end{eqnarray*}
The previous equality combined with (\ref{Yn}) gives

$$ Y_{0}^{n}(0,0)=\E\left[ \int_{0}^{\tau_{0}^{n}} e^{-rs}h(s,L_{s})\,ds+ \Int_{\tau_{0}^{n}}^{\tau_{0}^{n}+\Delta}  e^{-rs}  h(s,L_{s})ds- e^{-r(\tau_{0}^{n}+\Delta)}\psi(\beta_{0}^{n})+Y_{\tau_{0}^{n}+\Delta}^{n-1}(\tau_{0}^{n},\beta_{0}^{n})\right].$$
Hence,
\begin{equation}\label{Yn00}
Y_{0}^{n}(0,0)=\E\left[ \int_{0}^{\tau_{0}^{n}+\Delta} e^{-rs}h(s,L_{s})\,ds- e^{-r(\tau_{0}^{n}+\Delta)}\psi(\beta_{0}^{n})+Y_{\tau_{0}^{n}+\Delta}^{n-1}(\tau_{0}^{n},\beta_{0}^{n})\right].
\end{equation}
By using (\ref{ordrenn}) again, we obtain
$$ Y_{\tau_{0}^{n}+\Delta}^{n-1}(\tau_{0}^{n},\beta_{0}^{n})=\esssup\limits_{\tau \ge
{\tau_{0}^{n}+\Delta}}\E\left[ \int_{\tau_{0}^{n}+\Delta}^{\tau}e^{-rs}h(s,L_{s}+\beta_{0}^{n})\,ds+ O_{ \tau}^{n-1}(\tau_{0}^{n},\beta_{0}^{n})|\mathcal{F}_{\tau_{0}^{n}+\Delta}\right],$$
and $\tau_{1}^{n}$ is an optimal stopping time after $\tau_{0}^{n}+\Delta $. Then,
\begin{eqnarray*}
Y_{\tau_{0}^{n}+\Delta}^{n-1}(\tau_{0}^{n},\beta_{0}^{n})&=&\E \bigg\lbrack \int_{\tau_{0}^{n}+\Delta}^{\tau_{1}^{n}}e^{-rs}h(s,L_{s}+\beta_{0}^{n})\,ds+ O_{ \tau_{1}^{n}}^{n-1}(\tau_{0}^{n},\beta_{0}^{n})|\mathcal{F}_{\tau_{0}^{n}+\Delta}\bigg\rbrack \\
&=& \E \bigg\lbrack \int_{\tau_{0}^{n}+\Delta}^{\tau_{1}^{n} }e^{-rs}h(s,L_{s}+\beta_{0}^{n})\,ds + \E \bigg\lbrack \int_{\tau_{1}^{n}}^{\tau_{1}^{n}+\Delta}e^{-rs}h(s,L_{s}+\beta_{0}^{n})\,ds |\mathcal{F}_{\tau_{1}^{n}}\bigg\rbrack\\ &+&  \E \bigg\lbrack - e^{-r(\tau_{1}^{n}+\Delta)}\psi(\beta_{1}^{n})+Y_{\tau_{1}^{n}+\Delta}^{n-2}(\tau_{1}^{n},\beta_{0}^{n}+\beta_{1}^{n})|\mathcal{F}_{\tau_{1}^{n}} \bigg\rbrack |\mathcal{F}_{\tau_{0}^{n}+\Delta} \bigg\rbrack.
\end{eqnarray*}
Therefore,
\begin{align}\label{eq518}
Y_{\tau_{0}^{n}+\Delta}^{n-1}(\tau_{0}^{n},\beta_{0}^{n}) &=\E \bigg\lbrack \int_{\tau_{0}^{n}+\Delta}^{\tau_{1}^{n}+\Delta}e^{-rs}h(s,L_{s}+\beta_{0}^{n})\,ds - e^{-r(\tau_{1}^{n}+\Delta)}\psi(\beta_{1}^{n})\nonumber\\ &+ Y_{\tau_{1}^{n}+\Delta}^{n-2}(\tau_{1}^{n},\beta_{0}^{n}+\beta_{1}^{n})|\mathcal{F}_{\tau_{0}^{n}+\Delta} \bigg\rbrack.
\end{align}
Now, inserting \eqref{eq518} in \eqref{Yn00}, we obtain
\begin{eqnarray*}
Y_{0}^{n}(0,0)& =&\E \bigg\lbrack \int_{0}^{\tau_{0}^{n}+\Delta} e^{-rs}h(s,L_{s})\,ds +\int_{\tau_{0}^{n}+\Delta}^{\tau_{1}^{n}+\Delta}e^{-rs}h(s,L_{s}+\beta_{0}^{n})\,ds \\ & - & e^{-r(\tau_{0}^{n}+\Delta)}\psi(\beta_{0}^{n}) - e^{-r(\tau_{1}^{n}+\Delta)}\psi(\beta_{1}^{n})
+ Y_{\tau_{1}^{n}+\Delta}^{n-2}(\tau_{1}^{n},\beta_{0}^{n}+\beta_{1}^{n}) \bigg\rbrack.
\end{eqnarray*}
Repeat this reasoning as many times as necessary to obtain
\begin{equation}\label{eq319}\begin{array}{lll}
Y_{0}^{n}(0,0)& =& \E \bigg\lbrack \int_{0}^{\tau_{0}^{n}+\Delta} e^{-rs}h(s,L_{s})\,ds +\sum_{1\leq k \leq n-1}\int_{\tau_{k-1}^{n}+\Delta}^{\tau_{k}^{n}+\Delta}e^{-rs}h(s,L_{s}+\beta_{0}^{n}+\cdots+\beta_{k-1}^{n})\,ds \\ & -& \sum_{k=0}^{n-1}e^{-r(\tau_{k}^{n}+\Delta)}\psi(\beta_{k}^{n})+ Y_{\tau_{n-1}^{n}+\Delta}^{0}(\tau_{n-1}^{n},\beta_{0}^{n}+\cdots+\beta_{n-1}^{n})\bigg\rbrack.\end{array}
\end{equation}
Next, in view of  (\ref{Y0}),  we have
$$
Y_{\tau_{n-1}^{n}+\Delta}^{0}(\tau_{n-1}^{n},\beta_{0}^{n}+\cdots+\beta_{n-1}^{n})=\E\left[\int_{\tau_{n-1}^{n}+\Delta}^{\infty}e^{-rs}h(s,L_{s}+\beta_{0}^{n}+\cdots+\beta_{n-1}^{n})\,ds|\Fc_{\tau_{n-1}^{n}+\Delta}\right].
$$
By inserting  the last term in \eqref{eq319},  we obtain
\begin{eqnarray*}
Y_{0}^{n}(0,0)& =& \E \bigg\lbrack \int_{0}^{\tau_{0}^{n}+\Delta} e^{-rs}h(s,L_{s})ds + \sum_{ k\geq 1 }\int_{\tau_{k-1}^{n}+\Delta }^{\tau_{k}^{n}+\Delta} e^{-rs}h(s,L_{s}+\beta_{0}^{n}+ \cdots +\beta_{k-1}^{n})\,ds
 \\ &  -& \sum_{k \geq 0}e^{-r(\tau_{k}^{n}+\Delta)} \psi(\beta_{k}^{n}) \bigg\rbrack
 =J(\delta_{n}^{*}),
\end{eqnarray*}
where we have set $\tau_n^n=+\infty$, $\P$-a.s.

Next, it remains to show that the strategy $\delta_{n}^{*} $ is optimal over $\mathcal{A}_{n}$, i.e., $J(\delta_{n}^{*})\geq J(\delta_{n}^{'}) $ for any $\delta_{n}^{\prime}\in \mathcal{A}_{n}$. Indeed, let  $\delta_{n}^{\prime}=(\tau_{k}^{\prime},\beta_{k}^{\prime})_{k\ge 0}$ be a strategy of $\mathcal{A}_{n}$ (then $\tau_{n}^{\prime}=+\infty$, $\P$-a.s.). The definition of the Snell envelope allows us to write
\begin{equation} Y_{0}^{n}(0,0)\geq \E \left[ \int_{0}^{\tau_{0}^{\prime n}}e^{-rs}h(s,L_{s})\,ds+ O_{\tau_{0}^{\prime n}}^{n}(0,0)\right],
\end{equation}
where
\begin{eqnarray*}
O_{\tau_{0}^{\prime n}}^{n}(0,0) &=&\E \left[ \int_{\tau_{0}^{\prime n}}^{\tau_{0}^{\prime n}+\Delta}e^{-rs}h(s,L_{s})\,ds| \Fc_{\tau_{0}^{\prime n}} \right]\\ & +& \max \limits_{\beta \in U}\left\lbrace \E \left[ -e^{-r(\tau_{0}^{\prime n}+\Delta)}\psi(\beta)+ Y_{\tau_{0}^{ \prime n}+\Delta}^{n-1}(\tau_{0}^{\prime n},\beta)|\Fc_{\tau_{0}^{\prime n}}\right]\right\rbrace
\end{eqnarray*}
since
$Y_{\tau_{0}^{ \prime n}}^{n-1}(0,\beta)=Y_{\tau_{0}^{ \prime n}}^{n-1}(\tau_{0}^{\prime n},\beta)$ for any
$\beta \in U$. Next, by \eqref{determinynn} we have
\begin{align}\label{inside2}
&\nn \E \left[ -
e^{-r(\tau_{0}^{\prime n}+\Delta)}\psi(\beta_{0}^{\prime n})+  Y_{\tau_{0}^{ \prime n}+\Delta}^{n-1}(\tau_{0}^{\prime n},\beta_{0}^{\prime n})|\Fc_{\tau_{0}^{\prime n}}\right]\\\nn &\qquad =\sum_{\theta\in U }
\1_{\{\beta_{0}^{\prime n}=\theta\}}\E \left[ -
e^{-r(\tau_{0}^{\prime n}+\Delta)}\psi(\theta)+  Y_{\tau_{0}^{ \prime n}+\Delta}^{n-1}(\tau_{0}^{\prime n},\theta)|\Fc_{\tau_{0}^{\prime n}}\right]\\\nn &\qquad \le \sum_{\theta\in U }\1_{\{\beta_{0}^{\prime n}=\theta\}}
\max \limits_{\beta \in U}\left\lbrace \E \left[ -e^{-r(\tau_{0}^{\prime n}+\Delta)}\psi(\beta)+ Y_{\tau_{0}^{ \prime n}+\Delta}^{n-1}(\tau_{0}^{\prime n},\beta)|\Fc_{\tau_{0}^{\prime n}}\right]\right\rbrace\\
& \qquad =\max \limits_{\beta \in U}\left\lbrace \E \left[ -e^{-r(\tau_{0}^{\prime n}+\Delta)}\psi(\beta)+ Y_{\tau_{0}^{ \prime n}+\Delta}^{n-1}(\tau_{0}^{\prime n},\beta)|\Fc_{\tau_{0}^{\prime n}}\right]\right\rbrace.
\end{align}
Therefore
\begin{eqnarray*}
O_{\tau_{0}^{\prime n}}^{n}(0,0) \ge \E \left[ \int_{\tau_{0}^{\prime n}}^{\tau_{0}^{\prime n}+\Delta} e^{-rs}h(s,L_{s})ds -
e^{-r(\tau_{0}^{\prime n}+\Delta)}\psi(\beta_{0}^{\prime n})+  Y_{\tau_{0}^{ \prime n}}^{n-1}(\tau_{0}^{\prime n},\beta_{0}^{\prime n})|\Fc_{\tau_{0}^{\prime n}}\right]
\end{eqnarray*}
and then
 \begin{eqnarray*}
 Y_{0}^{n}(0,0)\geq \E \left[ \Int_{0}^{\tau_{0}^{\prime n+\Delta}}e^{-rs}h(s,L_{s})\,ds  -
e^{-r(\tau_{0}^{\prime n}+\Delta)}\psi(\beta_{0}^{\prime n})+  Y_{\tau_{0}^{\prime n}+\Delta}^{n-1}(\tau_{0}^{\prime n},\beta_{0}^{\prime n})\right].
\end{eqnarray*}
On the other hand, we have
\begin{eqnarray*}
Y_{\tau_{0}^{\prime n}+\Delta}^{n-1}(\tau_{0}^{\prime n},\beta_{0}^{\prime n})&=& \esssup \limits_{\tau \ge \tau_{0}^{\prime n}+\Delta } \E \left[ \int_{\tau_{0}^{\prime n}+\Delta}^{\tau}e^{-rs}h(s,L_{s}+\beta_{0}^{\prime n})\,ds + O^{n-1}_{\tau}(\tau_{0}^{\prime n},\beta_{0}^{\prime n})|\Fc_{\tau_{0}^{\prime n}+\Delta}\right] \\
& \geq & \E \left[ \int_{\tau_{0}^{\prime n}+\Delta}^{\tau_{1}^{\prime n}}e^{-rs}h(s,L_{s}+\beta_{0}^{\prime n})\,ds + O^{n-1}_{\tau_{1}^{\prime n}}(\tau_{0}^{\prime n},\beta_{0}^{\prime n})|\Fc_{\tau_{0}^{\prime n}+\Delta}\right] \\
 & \geq & \E \left[ \int_{\tau_{0}^{\prime n}+\Delta}^{\tau_{1}^{\prime n}+\Delta}e^{-rs}h(s,L_{s}+\beta_{0}^{\prime n})\,ds  -
e^{-r(\tau_{1}^{\prime n}+\Delta)}\psi(\beta_{1}^{\prime n})+  Y_{\tau_{1}^{\prime n}+\Delta}^{n-2}(\tau_{1}^{\prime n},\beta_{0}^{\prime n}+ \beta_{1}^{\prime n})|\Fc_{\tau_{0}^{\prime n}+\Delta}\right].
\end{eqnarray*}
This yields that
\begin{eqnarray*}
 Y_{0}^{n}(0,0) & \geq & \E \bigg\lbrack \Int_{0}^{\tau_{0}^{\prime n}+\Delta}e^{-rs}h(s,L_{s})\,ds + \int_{\tau_{0}^{\prime n}+\Delta}^{\tau_{1}^{\prime n}+\Delta}e^{-rs}h(s,L_{s}+\beta_{0}^{\prime n})\,ds \\
 &-& e^{-r(\tau_{0}^{\prime n}+\Delta)}\psi(\beta_{0}^{\prime n}) - e^{-r(\tau_{1}^{\prime n}+\Delta)}\psi(\beta_{1}^{\prime n})+  Y_{\tau_{1}^{\prime n}+\Delta}^{n-2}(\tau_{1}^{\prime n},\beta_{0}^{\prime n}+ \beta_{1}^{\prime n}) \bigg\rbrack.
\end{eqnarray*}
Repeat this reasoning as many times as necessary, we obtain
\begin{eqnarray*}
Y_{0}^{n}(0,0) & \geq & \E \big[ \Int_{0}^{\tau_{0}^{\prime n}+\Delta}e^{-rs}h(s,L_{s})\,ds +\sum_{1\leq k \leq n-1} \int_{\tau_{k-1}^{\prime n}+\Delta}^{\tau_{k}^{\prime n}+\Delta}e^{-rs}h(s,L_{s}+\beta_{0}^{\prime n}+\cdots + \beta_{k-1}^{\prime n} )\,ds \\  & + &  \int_{\tau_{n-1}^{\prime n}+\Delta}^{ + \infty }e^{-rs}h(s,L_{s}+\beta_{0}^{\prime n}+\cdots + \beta_{n-1}^{\prime n} )\,ds  - \sum_{k=0}^{n-1} e^{-r(\tau_{k}^{\prime n}+\Delta)}\psi (\beta_{k}^{\prime n})\big]
\\ &=& J(\delta_{n}^{\prime}).
\end{eqnarray*}
Hence,
$$ Y_{0}^{n}(0,0)= J(\delta_{n}^{*})\geq J(\delta_{n}^{\prime}), $$
which implies that strategy $\delta_{n}^{*} $ is optimal .
\end{proof}
\section{Impulse control problem in the general case }
In this section we consider the case when the number of
interventions is not limited, i.e., the controller can intervene as
many times as she wishes. In this case, existence of the optimal
control over all admissible strategies, heavily relies on the
continuity of the limiting process $(Y_{t}(\nu,\xi))_{t\geq 0}$
which is a crucial property of the value function.
\begin{Proposition}\label{Proposition_cont} The process $(Y_{t}(\nu,\xi))_{t\geq 0}$ given by \eqref{limit} is continuous.
\end{Proposition}
\begin{proof}
First, note that the process $(O_{t}(\nu,\xi))_{t \geq 0}$ is c\`adl\`ag since $Y(\nu,\xi)$ is so by (i) of Proposition \ref{Proposition2} and Appendix, Part (II). Next, let $T$ be a predictable stopping time such that $\Delta_T Y(\nu,\xi):=Y_T(\nu,\xi)-Y_{T^{-}}(\nu,\xi)<0$. By Part (I)-iii) of the Appendix, the process
$(O_t(\nu,\xi))_{t\geq 0}$ has a negative jump at $T$ and
$O_{T^{-}}(\nu,\xi)=Y_{T^{-}}(\nu,\xi)$. We then have:
\begin{eqnarray*}
O_{T^{-}}(\nu,\xi)- O_{T}(\nu,\xi)&= & \max_{\beta \in U} \bigg\lbrace \E \bigg\lbrack -e^{-r(T+\Delta)}\psi(\beta)+Y_{(T+\Delta)^{-}}(\nu,\xi+\beta)|\Fc_{T}\bigg\rbrack \bigg\rbrace \\ & - & \max_{\beta \in U} \bigg\lbrace \E \bigg\lbrack -e^{-r(T+\Delta)}\psi(\beta)+Y_{T+\Delta}(\nu,\xi+\beta)|\Fc_{T}\bigg\rbrack \bigg\rbrace  \\
&\leq &
\max_{\beta \in U} \bigg\lbrace \E\bigg\lbrack Y_{(T+\Delta)^{-}}(\nu,\xi+\beta)-Y_{T+\Delta}(\nu,\xi+\beta)|\Fc_{T}\bigg\rbrack \bigg\rbrace \\   &= &
 \max_{\beta \in U} \bigg\lbrace \E \bigg\lbrack \1_{A_{T+\Delta}(\xi+\beta)}\bigg\lbrace Y_{(T+\Delta)^{-}}(\nu,\xi+\beta)-Y_{T+\Delta}(\nu,\xi+\beta)\bigg\rbrace |\Fc_{T} \bigg\rbrack,
\end{eqnarray*}
where for any predictable stopping time $T\geq \nu$ and $\xi$ an  $\Fc_\nu$-measurable r.v.,\\ $A_T(\xi):=\{\omega \in \Omega, \Delta_TY(\nu,\xi)<0\}$ which belongs to $\mathcal{F}_T$. Thus
\begin{eqnarray}\label{contx1} \nonumber
\1_{A_{T}(\xi)}\{O_{T^{-}}(\nu,\xi)- O_{T}(\nu,\xi)\} &\leq &
 \max_{\beta \in U} \bigg\lbrace \E \bigg\lbrack \1_{A_{T}(\xi)}\times \1_{A_{T+\Delta}(\xi+\beta)} \bigg\lbrace Y_{(T+\Delta)^{-}}(\nu,\xi+\beta) \\ &-&Y_{T+\Delta}(\nu,\xi+\beta) \bigg\rbrace |\Fc_{T} \bigg\rbrack .
\end{eqnarray}
We note that there exists at least one $\beta \in U$ such that the right-hand side is positive. Otherwise the left-hand side is null and this is contradictory.
Since $Y_{T+\Delta}(\nu,\xi+\beta)\geq O_{T+\Delta}(\nu,\xi+\beta)$ and on the set $A_{T+\Delta}(\xi+\beta)$ it holds that $Y_{(T+\Delta)^{-}}(\nu,\xi+\beta)=O_{(T+\Delta)^{-}}(\nu,\xi+\beta)$, (\ref{contx1}) implies
\begin{eqnarray*}
\1_{A_{T}(\xi)}\{O_{T^{-}}(\nu,\xi)- O_{T}(\nu,\xi)\} &\leq &
 \max_{\beta \in U} \bigg\lbrace \E \bigg\lbrack \1_{A_{T}(\xi)}\times \1_{A_{T+\Delta}(\xi+\beta)} \bigg\lbrace O_{(T+\Delta)^{-}}(\nu,\xi+\beta) \\ &-&O_{T+\Delta}(\nu,\xi+\beta) \bigg\rbrace |\Fc_{T} \bigg\rbrack \\
 &\leq & \E \bigg\lbrack \1_{A_{T}(\xi)}\times  \max_{\beta \in U}\bigg\lbrace \E \bigg\lbrack \1_{A_{T+\Delta}(\xi+\beta)} \bigg\lbrace O_{(T+\Delta)^{-}}(\nu,\xi+\beta) \\ &-&O_{T+\Delta}(\nu,\xi+\beta) \bigg\rbrace|\Fc_{T+\Delta} \bigg\rbrack \bigg\rbrace |\Fc_{T} \bigg\rbrack \\
 &\leq & \E \bigg\lbrack \1_{A_{T}(\xi)}\times  \E \bigg\lbrack \1_{A_{T+\Delta}(\xi+\beta_{1})} \bigg\lbrace O_{(T+\Delta)^{-}}(\nu,\xi+\beta_{1}) \\ &-&O_{T+\Delta}(\nu,\xi+\beta_{1}) \bigg\rbrace|\Fc_{T+\Delta} \bigg\rbrack |\Fc_{T} \bigg\rbrack,
\end{eqnarray*}
where $\beta_1$ is a r.v. $\Fc_{T+\Delta}$-measurable valued in $U$. The r.v. $\beta_1$ can be constructed in the following way.
For $i=1,\ldots,p$, let $\Bc_i$ be the set,
\begin{align}
&\nn \Bc_i:=\bigg\lbrace
\max_{\beta \in U}  \E \bigg\lbrack \1_{A_{T+\Delta}(\xi+\beta)} \bigg\lbrace O_{(T+\Delta)^{-}}(\nu,\xi+\beta) -O_{T+\Delta}(\nu,\xi+\beta) \bigg\rbrace |\Fc_{T+\Delta} \bigg\rbrack
\\\nn &\qquad \qquad \qquad \qquad=
\E \bigg\lbrack \1_{A_{T+\Delta}(\xi+\theta_i)} \bigg\lbrace O_{(T+\Delta)^{-}}(\nu,\xi+\theta_i) -O_{T+\Delta}(\nu,\xi+\theta_i) \bigg\rbrace |\Fc_{T+\Delta} \bigg\rbrack \bigg\rbrace \end{align}
We now define $\beta_1$ as
\begin{equation}\label{inside3}
\beta_1=\theta_1 \mbox{ on } \Bc_1 \text{ and  }
\beta_1=\theta_j \mbox{ on
}\Bc_j\backslash\bigcup_{k=1}^{j-1}\Bc_k \text{ for }j=2,\ldots,p.
\end{equation}
Therefore,
\begin{align}
&\nn \max_{\beta \in U}  \E \bigg\lbrack \1_{A_{T+\Delta}(\xi+\beta)} \bigg\lbrace O_{(T+\Delta)^{-}}(\nu,\xi+\beta) -O_{T+\Delta}(\nu,\xi+\beta) \bigg\rbrace |\Fc_{T+\Delta} \bigg\rbrack \\&\nn \qquad\qquad
=\sum_{i=1,p}1_{\{\beta_1=\theta_i\}}
\max_{\beta \in U}  \E \bigg\lbrack \1_{A_{T+\Delta}(\xi+\beta)} \bigg\lbrace O_{(T+\Delta)^{-}}(\nu,\xi+\beta) -O_{T+\Delta}(\nu,\xi+\beta) \bigg\rbrace |\Fc_{T+\Delta} \bigg\rbrack\\&\nn \qquad\qquad
=\sum_{i=1,p}1_{\{\beta_1=\theta_i\}}\E \bigg\lbrack \1_{A_{T+\Delta}(\xi+\theta_i)} \bigg\lbrace O_{(T+\Delta)^{-}}(\nu,\xi+\theta_i) -O_{T+\Delta}(\nu,\xi+\theta_i) \bigg\rbrace |\Fc_{T+\Delta} \bigg\rbrack \\&\nn \qquad\qquad
=\E \bigg\lbrack\sum_{i=1,p}1_{\{\beta_1=\theta_i\}} \1_{A_{T+\Delta}(\xi+\theta_i)} \bigg\lbrace O_{(T+\Delta)^{-}}(\nu,\xi+\theta_i) -O_{T+\Delta}(\nu,\xi+\theta_i) \bigg\rbrace |\Fc_{T+\Delta} \bigg\rbrack\\&\nn \qquad\qquad
=\E \bigg\lbrack\sum_{i=1,p}1_{\{\beta_1=\theta_i\}} \1_{A_{T+\Delta}(\xi+\beta_1)} \bigg\lbrace O_{(T+\Delta)^{-}}(\nu,\xi+\beta_1) -O_{T+\Delta}(\nu,\xi+\beta_1) \bigg\rbrace |\Fc_{T+\Delta} \bigg\rbrack
\end{align}
since $A_{T+\Delta}(\xi+\theta_i)=\{Y_{(T+\Delta)-}(\nu,\xi+\theta_i)- Y_{(T+\Delta)-}(\nu, \xi+\theta_i)>0\}$ and by \eqref{determiny} on $\beta_1=\theta_i$, $Y_{(T+\Delta)-}(\nu,\xi+\theta_i)=Y_{(T+\Delta)-}(\nu,\xi+\beta_1)$,
$Y_{(T+\Delta)}(\nu,\xi+\theta_i)=Y_{(T+\Delta)}(\nu,\xi+\beta_1)$ and the same is valid for $O_{(T+\Delta)^{-}}(\nu,\xi+\theta_i)$ and $O_{T+\Delta}(\nu,\xi+\theta_i)$. Therefore
\begin{align}
&\nn \max_{\beta \in U}  \E \bigg\lbrack \1_{A_{T+\Delta}(\xi+\beta)} \bigg\lbrace O_{(T+\Delta)^{-}}(\nu,\xi+\beta) -O_{T+\Delta}(\nu,\xi+\beta) \bigg\rbrace |\Fc_{T+\Delta} \bigg\rbrack\\&\nn  \qquad\qquad
=\E \bigg\lbrack\1_{A_{T+\Delta}(\xi+\beta_1)} \bigg\lbrace O_{(T+\Delta)^{-}}(\nu,\xi+\beta_1) -O_{T+\Delta}(\nu,\xi+\beta_1) \bigg\rbrace |\Fc_{T+\Delta} \bigg\rbrack .
\end{align}
Repeating this reasoning $n$ times yields
\begin{eqnarray*}\label{cont3}
\1_{A_{T}(\xi)}\{O_{T^{-}}(\nu,\xi)& -& O_{T}(\nu,\xi)\} \leq
 \E \bigg\lbrack \1_{A_{T}(\xi)}\bigg\lbrace \prod_{k=1}^{k=n} 1_{A_{T+k\Delta}(\xi+\beta_1+ \cdots +\beta_k)} \\ &\times &\bigg( O_{(T+n\Delta)^{-}}(\nu,\xi+\beta_1+ \cdots +\beta_n)-O_{T+n\Delta}(\nu,\xi+\beta_1+ \cdots +\beta_n)\bigg) \bigg\rbrace |\Fc_{T}\bigg\rbrack,
\end{eqnarray*}
where the random variables $\beta_{k}$ are valued in $U$ and $\Fc_{T+k\Delta}$-measurable.
But the left-hand side converges to $0$, $\P$-a.s. when $n\rightarrow +\infty$. Indeed, by using (\ref{bornee}) for any $\nu$ and $\xi\in \Fc_\nu$, we have
\begin{equation}
 |O_t(\nu,\xi)|\leq  \frac{\gamma}{r}\{e^{-rt}-e^{-r(t+\Delta)}\}
+\|\psi\|e^{-r(t+\Delta)}+\frac{\gamma}{r} e^{-r(t+\Delta)}, \quad \forall \, t\geq 0,
\end{equation}
and then $\underset{t\to \infty}{\lim}O_t(\nu,\xi)=0$ uniformly with respect to $\nu$ and $\xi$. Thus
$$
\1_{A_{T}(\xi)}\{O_{T^{-}}(\nu,\xi)-O_{T}(\nu,\xi)\}=0,
$$
which is absurd. Hence, the process $Y(\nu,\xi)$ is continuous.
\end{proof}
\begin{Remark} Since the process $Y(\nu,\xi)$ is continuous and satisfies \eqref{limit}, then there exist processes
$Z(\nu,\xi)$ and $K(\nu,\xi))$ which belong respectively to $\Hc^{2,d}$ and $\mathcal{S}_{i}^{2}$ such that the triple $(Y(\nu,\xi),Z(\nu,\xi),K(\nu,\xi))$ satisfies the following reflected BSDE: $\forall t\ge 0$,
$$\left\{
\begin{array}{l}
 Y_{t}(\nu,\xi)=\int_{t}^{\infty}e^{-rs}h(s,L_{s}+\xi)\1_{[s\geq \nu]}\,ds+K_{\infty}(\nu,\xi)-K_{t}(\nu,\xi)-\int_{t}^{\infty}Z_{s}(\nu,\xi)\,dB_{s} \,;
\\
Y_{t}(\nu,\xi)  \geq  O_{t}(\nu,\xi) \text{ and }\int_{0}^{\infty}(Y_{t}(\nu,\xi)-O_{t}(\nu,\xi))\,dK_{t}(\nu,\xi)=0\nonumber
\end{array}\right.$$
where the process $O(\nu,\xi)$ is given by \eqref{limito}.\qed
\end{Remark}
We now give the main result of this section.
\begin{Theorem}\label{Theorem_cont} Let us assume that Assumption \ref{assumpt} hold and let us define the  strategy $\delta^{*}=(\tau_{n}^{*},\beta_{n}^{*})_{n\geq 0}$  by
$$\tau_{0}^{*}= \begin{cases}\inf\{ s\in[0,\infty), O_{s}(0,0)\geq Y_{s}(0,0)\},\\  + \infty, \quad \text{otherwise} \end{cases}$$ \\
and $\beta_{0}^{*}$ an $\mathcal{F}_{\tau_{0}^{*}}$-r.v. such that
$$ O_{\tau_{0}^{*}}(0,0):= \E\left[\Int_{\tau_{0}^{*}}^{\tau_{0}^{*}+\Delta}  e^{-rs}  h(s,L_{s})ds- e^{-r(\tau_{0}^{*}+\Delta)}\psi(\beta_{0}^{*})+Y_{\tau_{0}^{*}+\Delta}(\tau_{0}^{*},\beta_{0}^{*})|\mathcal{F}_{\tau_{0}^{*}}\right].$$
For any $n \geq 1$,
 $$ \tau_{n}^{*}=\inf\bigg\lbrace s\geq \tau_{n-1}^{*}+\Delta , O_{s}(\tau_{n-1}^{*},\beta_{0}^{*}+\cdots +\beta_{n-1}^{*} )\geq Y_{s}(\tau_{n-1}^{*},\beta_{0}^{*}+\cdots +\beta_{n-1}^{*} )\bigg\rbrace,$$
 and $\beta_{n}^{*}$ an $U$-valued $\mathcal{F}_{\tau_{n}^{*}}$-measurable r.v. such that
 \begin{eqnarray*}
 O_{\tau_{n}^{*}}(\tau_{n-1}^{*},\beta_{0}^{*}+\cdots +\beta_{n-1}^{*} )&=&\E \bigg\lbrack \Int_{\tau_{n}^{*}}^{\tau_{n}^{*}+\Delta}  e^{-rs}  h(s,L_{s}+\beta_{0}^{*}+\cdots +\beta_{n-1}^{*} )ds \\ &-& e^{-r(\tau_{n}^{*}+\Delta)}\psi(\beta_{n}^{*})+Y_{\tau_{n}^{*}+\Delta}(\tau_{n}^{*},\beta_{0}^{*}+\cdots +\beta_{n-1}^{*} +\beta_{n}^{*})|\mathcal{F}_{\tau_{n}^{*}}\bigg\rbrack.
 \end{eqnarray*}
 Then, the strategy $\delta^{*}=(\tau_{n}^{*},\beta_{n}^{*})_{n\geq 0}$ is optimal for the impulse control problem, i.e.,
 $$ Y_{0}(0,0)=\Sup_{\delta \in \mathcal{A}}J(\delta)=J(\delta^{*}).$$
\end{Theorem}
\begin{proof}
We first prove that  $ Y_{0}(0,0)= J(\delta^{*}) $.\\

\noindent We have:
\begin{equation}
 Y_{0}(0,0)= \esssup\limits_{\tau \in
\mathcal{T}_0}\E\left[\int_{0}^{\tau} e^{-rs}h(s,L_{s})
ds+  O_{\tau}(0,0)\right].
\end{equation}
Since $Y(\nu,\xi)$ and $(O_t(0,0))_{t\ge 0}$ are continuous on $[0,\infty]$, then,
for any stopping time $\nu$ and any $\mathcal{F}_{\nu}$-measurable
r.v. $\xi$, the stopping time $\tau_{0}^{*}$ is optimal after $0$.
This yields
\begin{equation}\label{equationY0}
 Y_{0}(0,0)= \E\left[\int_{0}^{\tau_{0}^{*}} e^{-rs}h(s,L_{s})
ds+  O_{\tau_{0}^{*}}(0,0)\right]
\end{equation}
where
\begin{eqnarray*}
O_{\tau_{0}^{*}}(0,0) &=&\E\left[ \Int_{\tau_{0}^{*}}^{\tau_{0}^{*}+\Delta}  e^{-rs}  h(s,L_{s})ds|\mathcal{F}_{\tau_{0}^{*}}\right]\\ &+& \max\limits_{\beta\in U} \left\lbrace \E \left[e^{-r(\tau_{0}^{*}+\Delta)}(-\psi(\beta))+Y_{\tau_{0}^{*}+\Delta}(0,\beta)|\mathcal{F}_{\tau_{0}^{n}} \right] \right\rbrace \\
&=& \E\left[\Int_{\tau_{0}^{*}}^{\tau_{0}^{*}+\Delta}  e^{-rs}  h(s,L_{s})ds- e^{-r(\tau_{0}^{*}+\Delta)}\psi(\beta_{0}^{*})+  Y_{\tau_{0}^{*}+\Delta}(\tau_{0}^{*},\beta_{0}^{*})|\mathcal{F}_{\tau_{0}^{*}}\right].
\end{eqnarray*}
Note that the second equality is valid thanks to Proposition
\ref{Proposition2}-ii) since
$Y_{\tau_{0}^{*}+\Delta}(0,\beta)=Y_{\tau_{0}^{*}+\Delta}(\tau_{0}^{*},\beta)$,
for all $\beta \in U$. Combining this with (\ref{equationY0}), we
obtain
\begin{eqnarray*}
 Y_{0}(0,0) &=& \E\left[\int_{0}^{\tau_{0}^{*}} e^{-rs}h(s,L_{s})
ds+ \Int_{\tau_{0}^{*}}^{\tau_{0}^{*}+\Delta}  e^{-rs}  h(s,L_{s})ds  -  e^{-r(\tau_{0}^{*}+\Delta)}\psi(\beta_{0}^{*})+  Y_{\tau_{0}^{*}+\Delta}(\tau_{0}^{*},\beta_{0}^{*})\right] \\
&=& \E\left[\int_{0}^{\tau_{0}^{*}+\Delta} e^{-rs}h(s,L_{s}) ds -  e^{-r(\tau_{0}^{*}+\Delta)}\psi(\beta_{0}^{*})+  Y_{\tau_{0}^{*}+\Delta}(\tau_{0}^{*},\beta_{0}^{*})\right].
\end{eqnarray*}
On the other hand, we have that
$$ Y_{\tau_{0}^{*}+\Delta}(\tau_{0}^{*},\beta_{0}^{*})= \esssup\limits_{\tau \in
\mathcal{T}_{\tau_{0}^{*}+\Delta}}\E\left[\int_{\tau_{0}^{*}+\Delta}^{\tau} e^{-rs}h(s,L_{s}+\beta_{0}^{*})
ds+  O_{\tau}(\tau_{0}^{*},\beta_{0}^{*})|\mathcal{F}_{\tau_{0}^{*}+\Delta}\right]. $$
As the stopping time $\tau_{1}^{*}$ is optimal after $\tau_{0}^{*}+\Delta$, then
\begin{eqnarray*}
Y_{\tau_{0}^{*}+\Delta}(\tau_{0}^{*},\beta_{0}^{*})& = &
\E\left[\int_{\tau_{0}^{*}+\Delta}^{\tau_{1}^{*}} e^{-rs}h(s,L_{s}+\beta_{0}^{*})
ds+  O_{\tau_{1}^{*}}(\tau_{0}^{*},\beta_{0}^{*})|\mathcal{F}_{\tau_{0}^{*}+\Delta}\right] \\
& =& \E\left[\int_{\tau_{0}^{*}+\Delta}^{\tau_{1}^{*}+\Delta} e^{-rs}h(s,L_{s}+\beta_{0}^{*}) ds -  e^{-r(\tau_{1}^{*}+\Delta)}\psi(\beta_{1}^{*})+  Y_{\tau_{1}^{*}+\Delta}(\tau_{1}^{*},\beta_{0}^{*}+\beta_{1}^{*})|\mathcal{F}_{\tau_{0}^{*}+\Delta}\right].
\end{eqnarray*}
We insert this last quantity in the previous one to obtain
\begin{eqnarray*}
 Y_{0}(0,0) &=& \E \bigg\lbrack \int_{0}^{\tau_{0}^{*}+\Delta} e^{-rs}h(s,L_{s}) ds + \int_{\tau_{0}^{*}+\Delta}^{\tau_{1}^{*}+\Delta} e^{-rs}h(s,L_{s}+\beta_{0}^{*}) ds \\ & -&   e^{-r(\tau_{0}^{*}+\Delta)}\psi(\beta_{0}^{*})-e^{-r(\tau_{1}^{*}+\Delta)}\psi(\beta_{1}^{*})+  Y_{\tau_{1}^{*}+\Delta}(\tau_{1}^{*},\beta_{0}^{*}+\beta_{1}^{*})\bigg\rbrack.
\end{eqnarray*}
Now, we use the same reasoning as many times as necessary to get
\begin{eqnarray} \nonumber
Y_{0}(0,0) &= &\E \bigg\lbrack \int_{0}^{\tau_{0}^{*}+\Delta} e^{-rs}h(s,L_{s}) ds  + \sum_{1\leq k \leq n-1} \int_{\tau_{k-1}^{*}+\Delta}^{\tau_{k}^{*}+\Delta} e^{-rs}h(s,L_{s}+\beta_{0}^{*}+\cdots+\beta_{k-1}^{*}) ds \\ \qquad &-&  \sum_{k=0}^{n-1} e^{-r(\tau_{k}^{*}+\Delta)}\psi(\beta_{k}^{*})+Y_{\tau_{n}^{*}+\Delta}(\tau_{n}^{*},\beta_{0}^{*}+\cdots + \beta_{n}^{*})\bigg\rbrack.
\end{eqnarray}
But, by \eqref{BorneY}, $\underset{n\to\infty}{\lim}Y_{\tau_{n}^{*}+\Delta}(\tau_{n}^{*},\beta_{0}^{*}+\cdots + \beta_{n}^{*})=0$. Thus,  take the limit w.r.t $n$ in the left hand-side of the previous equality to obtain that,
$$
Y_{0}(0,0)= J(\delta^{*}).
$$

To proceed, we prove that the strategy $\delta^{*}=(\tau_{n}^{*},\beta_{n}^{*})_{n\geq 0}$ is optimal for the general impulse control problem, i.e. $J(\delta^{*})\geq J(\delta ^{\prime})$ for any $\delta^{\prime}
=(\tau^\prime_{n},\beta^\prime_{ n})_{n\geq 0}$ in $\mathcal{A}$. The definition of the Snell envelope allows us to write
 $$ Y_{0}(0,0)\geq \E\left[\int_{0}^{\tau_{0}^{\prime}} e^{-rs}h(s,L_{s})
ds+  O_{\tau_{0}^{\prime}}(0,0)\right].
$$
But, we have
 $$ O_{\tau_{0}^{\prime}}(0,0) \geq \E\left[\Int_{\tau_{0}^{\prime}}^{\tau_{0}^{\prime}+\Delta}  e^{-rs}  h(s,L_{s})ds- e^{-r(\tau_{0}^{\prime}+\Delta)}\psi(\beta_{0}^{\prime})+  Y_{\tau_{0}^{\prime}+\Delta}(\tau_{0}^{\prime},\beta_{0}^{\prime})|\mathcal{F}_{\tau_{0}^{\prime}}\right]$$
which yields
\begin{eqnarray}
 Y_{0}(0,0)\geq \E\left[\int_{0}^{\tau_{0}^{\prime}+\Delta} e^{-rs}h(s,L_{s})ds - e^{-r(\tau_{0}^{\prime}+\Delta)}\psi(\beta_{0}^{\prime})+  Y_{\tau_{0}^{\prime}+\Delta}(\tau_{0}^{\prime},\beta_{0}^{\prime})\right].
\end{eqnarray}
Next, as in \eqref{inside2}, We have
\begin{eqnarray*}
Y_{\tau_{0}^{\prime}+\Delta}(\tau_{0}^{\prime},\beta_{0}^{\prime})& =& \esssup\limits_{\tau \in \mathcal{T}_{\tau_{0}^{\prime}+\Delta}}\E \left[
\int_{\tau_{0}^{\prime}+\Delta}^{\tau} e^{-rs}h(s,L_{s}+\beta_{0}^{\prime})ds + O_{\tau}(\tau_{0}^{\prime},\beta_{0}^{\prime})|\Fc_{\tau_{0}^{\prime}+\Delta}\right]\\
& \geq & \E \left[ \int_{\tau_{0}^{\prime}+\Delta}^{\tau_{1}^{\prime}} e^{-rs}h(s,L_{s}+\beta_{0}^{\prime})ds+ O_{\tau_{1}^{\prime}}(\tau_{0}^{\prime},\beta_{0}^{\prime})|\Fc_{\tau_{0}^{\prime}+\Delta}\right]\\
&\geq &\E \left[ \int_{\tau_{0}^{\prime}+\Delta}^{\tau_{1}^{\prime}+\Delta} e^{-rs}h(s,L_{s}+\beta_{0}^{\prime})ds-  e^{-r(\tau_{1}^{\prime}+\Delta)}\psi(\beta_{1}^{\prime})+  Y_{\tau_{1}^{\prime}+\Delta}(\tau_{1}^{\prime},\beta_{0}^{\prime}+\beta_{1}^{\prime})|\Fc_{\tau_{0}^{\prime}+\Delta}\right].
\end{eqnarray*}
Therefore,
\begin{eqnarray*}
Y_{0}(0,0) & \geq & \E \bigg\lbrack \int_{0}^{\tau_{0}^{\prime}+\Delta} e^{-rs}h(s,L_{s})ds+ \int_{\tau_{0}^{\prime}+\Delta}^{\tau_{1}^{\prime}+\Delta} e^{-rs}h(s,L_{s}+\beta_{0}^{\prime})ds \\ &-&  e^{-r(\tau_{0}^{\prime}+\Delta)}\psi(\beta_{0}^{\prime})- e^{-r(\tau_{1}^{\prime}+\Delta)}\psi(\beta_{1}^{\prime})+  Y_{\tau_{1}^{\prime}+\Delta}(\tau_{1}^{\prime},\beta_{0}^{\prime}+\beta_{1}^{\prime})\bigg\rbrack.
\end{eqnarray*}
By repeating this argument $n$ times, we obtain
\begin{eqnarray*}
Y_{0}(0,0) & \geq & \E \bigg\lbrack \int_{0}^{\tau_{0}^{\prime}+\Delta} e^{-rs}h(s,L_{s})ds+\sum_{1\leq k \leq n-1}\int_{\tau_{k-1}^{\prime}+\Delta}^{\tau_{k}^{\prime}+\Delta} e^{-rs}h(s,L_{s}+\beta_{0}^{\prime}+\cdots + \beta_{k-1}^{\prime})ds \\ &-& \sum_{k=0}^{n}e^{-r(\tau_{k}^{\prime}+\Delta)}\psi(\beta_{k}^{\prime})+ Y_{\tau_{n}^{\prime}+\Delta}(\tau_{n}^{\prime},\beta_{0}^{\prime}+\cdots +\beta_{n}^{\prime})\bigg\rbrack.
\end{eqnarray*}
Finally, taking the limit as $n\to +\infty$, yields
$$ Y_{0}(0,0) \geq  \E \left[\int_{0}^{+\infty } e^{-rs}h(s,L_{s}^{\delta^{\prime}})ds - \sum_{n\geq 0}e^{-r(\tau_{n}^{\prime}+\Delta)}\psi(\beta_{n}^{\prime})\right]= J(\delta^{\prime})$$since
$\lim_{n\rightarrow \infty}Y_{\tau_{n}^{\prime}+\Delta}(\tau_{n}^{\prime},\beta_{0}^{\prime}+\cdots +\beta_{n}^{\prime})=0$. Hence, the strategy $\delta^{*}$ is optimal.
 \end{proof}
\section{Risk-sensitive impulse control problem}
In this section, we extend the previous results to the risk-sensitive case where the controller has a utility function which is of exponential type. In order to tackle this problem we do not use BSDEs, as in the previous section, but instead, the Snell envelope notion which is more appropriate.  A similar version of this problem is considered in Hdhiri {\it et al.} \cite{Hdhiri} in the case when the horizon is finite.

When the decision maker implements a strategy $\delta=(\tau_n,\xi_n)_{n\ge 1}$, the payoff is given by
\begin{equation}
J(\delta):=\E \left[\exp\theta\left\lbrace\int_{0}^{\infty} e^{-rs}h(s,L_{s}^{\delta})\;ds-\sum_{n\geq 1}e^{-r(\tau_{n}+\Delta)}\psi(\xi_{n})\right\rbrace\right], \label{reward2}
\end{equation}
where $\theta>0$ is the risk-sensitive parameter. Hereafter, for sake of simplicity, we will treat only the case $\theta=1$ since the other cases are treated in a similar way.

We proceed by recasting the risk-sensitive impulse control problem into an iterative optimal stopping problem, and by exploiting the Snell envelope properties, we shall be able to characterize recursively an optimal strategy to this risk-sensitive impulse control problem.
\subsection{Iterative optimal stopping and properties}
 Let  $\nu$ be a stopping time and $\xi$ an $\Fc_{\nu}$-measurable random variable, we introduce the sequence of processes  $(Y^{n}(\nu, \xi))_{n\geq 0}$  defined recursively by
\begin{equation}\label{Y0app}
 Y_{t}^{0}(\nu,\xi)= \mathbb{E}\left[\exp\left\lbrace \int_{t}^{+\infty}e^{-rs}h(s,L_{s}+\xi)\1_{[s\geq \nu]}ds\right\rbrace|\Fc_{t}\right],\, t\geq 0,
\end{equation}
and, for $n\geq 1$,
\begin{equation} \label{Yapp}
Y_{t}^{n}(\nu,\xi)= \esssup\limits_{\tau\in \mathcal{T}_{t}}\mathbb{E}\left[\exp\left\lbrace \int_{t}^{\tau}e^{-rs}h(s,L_{s}+\xi)\1_{[s\geq \nu]}ds\right\rbrace O_{\tau}^{n}(\nu,\xi)|\Fc_{t}\right],\,\, t\geq 0,
\end{equation}
where
\begin{eqnarray*}
O_{t}^{n}(\nu,\xi)= \max\limits_{\beta \in U} \bigg\lbrace  \mathbb{E}\bigg\lbrack \exp\bigg\lbrace \int_{t}^{t+\Delta}e^{-rs}h(s,L_{s}+\xi)\1_{[s\geq \nu]}ds -e^{-r(t+\Delta)}\psi(\beta) \bigg\rbrace  \times  Y_{t+\Delta}^{n-1}(\nu,\xi+\beta)|\Fc_{t}\bigg\rbrack \bigg\rbrace.
\end{eqnarray*}
Then the sequence of  processes  $(Y^{n}(\nu, \xi))_{n\geq 0}$ enjoys the following properties.
\begin{Proposition}\lb{75}${}$

\noindent i)  For any $n \in\N$, the process  $Y^{n}(\nu, \xi)$ belongs to $\Sc^{2}_{c}$ and satisfies  $\lim\limits_{t\to +\infty}Y_{t}^{n}(\nu, \xi)=1$.\\
ii) The sequence of processes  $(Y^{n}(\nu, \xi))_{n\geq 0}$ satisfies, $\P$.a.s, for any $t\ge 0$,
\begin{equation}\lb{ineqyn}
0\le Y_{t}^{n}(\nu, \xi) \leq Y_{t}^{n+1}(\nu, \xi)\leq \exp(\frac{\gamma e^{-rt}}{r}).
\end{equation}
Moreover, the process $Y_{t}(\nu, \xi)=\lim_{n\rightarrow \infty}Y_{t}^{n}(\nu, \xi)$, $t\ge 0$, is c{\`a}dl{\` a}g and satisfies
\begin{equation}\lb{estimyexp}\P\mbox{-}a.s.\quad \forall t\ge 0,\quad 0\le  Y_{t}(\nu, \xi)\leq \exp(\frac{\gamma e^{-rt}}{r}).\end{equation}
Finally, it holds that
\begin{equation}\label{MDynamic}
Y_{t}(\nu,\xi)= \esssup\limits_{\tau\in \mathcal{T}_{t}}\mathbb{E}\left[\exp\left\lbrace \int_{t}^{\tau} e^{-rs} h(s,L_{s}+\xi)\1_{[s\geq \nu]}ds\right\rbrace O_{\tau}(\nu,\xi)|\Fc_{t}\right],
\end{equation}
where  $$ O_{t}(\nu,\xi):= \max\limits_{\beta \in U} \bigg\lbrace  \mathbb{E}\bigg\lbrack \exp\bigg \lbrace \int_{t}^{t+\Delta}e^{-rs}h(s,L_{s}+\xi) \1_{[s\geq \nu]}ds -e^{-r(t+\Delta)}\psi(\beta) \bigg\rbrace Y_{t+\Delta}(\nu,\xi+\beta)|\Fc_{t}\bigg\rbrack \bigg\rbrace.$$
iii)  For any two stopping times $\nu$ and $\nu^{\prime}$ such that $\nu\leq \nu^{\prime}$ and $\xi$ an $\Fc_{\nu}$-measurable r.v., we have
$$\P-a.s.,\,\,\forall t\ge \nu',\,\, Y_{t}(\nu, \xi)=Y_{t}(\nu^{\prime}, \xi). $$
\end{Proposition}
\begin{proof}  Let $\nu$ be a stopping time and $\xi$ an $\mathcal{F}_{\nu}$-measurable random variable.
\medskip

\noindent i) We will show by induction that for each $n\geq 0$, for any $\xi\in \Fc_\nu$, $Y^{n}(\nu, \xi)$ belongs to $\Sc^{2}_{c}$, satisfies  $\lim\limits_{t\to +\infty}Y_{t}^{n}(\nu, \xi)=1$ and $\P$-a.s, for any $t\ge 0$,
\begin{equation*}
0\le Y_{t}^{n}(\nu, \xi) \leq \exp(\frac{\gamma e^{-rt}}{r}).
\end{equation*}
Let us start with the case $n=0$. In view of the definition of $Y^{0}(\nu,\xi)$ given by (\ref{Y0app}), we have $\lim\limits_{t\to +\infty}Y_{t}^{0}(\nu, \xi)=1$ since $h$ is bounded. On the other hand,
\begin{eqnarray*}
 \mathbb{E}\left[\sup\limits_{t\geq 0} |Y_{t}^{0}(\nu,\xi)|^{2} \right] &=& \mathbb{E}\left[\sup\limits_{t\geq 0} \left| \mathbb{E}\left[ \exp\left\lbrace \int_{t}^{+\infty}e^{-rs}h(s,L_{s}+\xi)\1_{[s\geq \nu]}ds\right\rbrace|\Fc_{t}\right] \right|^{2} \right] \\&\leq& \mathbb{E}\left[\sup\limits_{t\geq 0} \exp\left\lbrace 2 \int_{t}^{+\infty} \gamma  e^{-rs} ds\right\rbrace\right]=
 \mathbb{E}\left[\exp\left\lbrace 2 \int_{0}^{+\infty} \gamma  e^{-rs} ds\right\rbrace \right] = \exp(2\frac{\gamma}{r}),
\end{eqnarray*}
since $h$ is uniformly  bounded by $\gamma$ (Assumption \ref{assumpt}). In addition, we note that  for every $t\ge 0$,
\begin{eqnarray*}
 Y_{t}^{0}(\nu,\xi )&=& \mathbb{E}\left[\exp\left\lbrace \int_{t}^{+\infty}e^{-rs}h(s,L_{s}+\xi)\1_{[s\geq \nu]}ds\right\rbrace|\Fc_{t}\right] \\
 &=& \mathbb{E}\left[\exp\left\lbrace \int_{0}^{+\infty}e^{-rs}h(s,L_{s}+\xi)\1_{[s\geq \nu]}ds\right\rbrace|\Fc_{t}\right] \exp\left\lbrace -\int_{0}^{t}e^{-rs}h(s,L_{s}+\xi)\1_{[s\geq \nu]}ds\right\rbrace.
\end{eqnarray*}
As martingales w.r.t. the Brownian filtration are continuous,  then
clearly $ Y^{0}(\nu,\xi )$ is continuous on $[0,+\infty]$, and then $ Y^{0}(\nu,\xi )$
belongs to $\Sc_{c}^{2}$. Finally
\begin{eqnarray*}
 0\le Y_{t}^{0}(\nu,\xi)=\mathbb{E}\left[\exp\left\lbrace \int_{t}^{+\infty}e^{-rs}h(s,L_{s}+\xi)\1_{[s\geq \nu]}ds\right\rbrace|\Fc_{t}\right] \le \exp\left\lbrace  \int_{t}^{+\infty} \gamma  e^{-rs} ds\right\rbrace=\exp(\frac{\gamma e^{-rt}}{r}).
\end{eqnarray*}
Thus the property holds for $n=0$. Assume now that it holds for some $n\ge 1$.
First note that since for every $t\ge 0$ and every $\xi\in \Fc_\nu$, $0 \leq Y^{n}_{t}(\nu,\xi)\leq \exp(\frac{\gamma e^{-rt}}{r}) $, then
\begin{eqnarray*}
0\le Y_{t}^{n+1}(\nu,\xi) &=& \esssup\limits_{\tau\in \mathcal{T}_{t}}\mathbb{E}\bigg\lbrack \exp\bigg\lbrace \int_{t}^{\tau}e^{-rs}h(s,L_{s}+\xi)\1_{[s\geq \nu]}ds\bigg\rbrace O_{\tau}^{n+1}(\nu,\xi)|\Fc_{t}\bigg\rbrack \\
&=& \esssup\limits_{\tau\in \mathcal{T}_{t}}\mathbb{E}\bigg\lbrack \exp\bigg\lbrace \int_{t}^{\tau}e^{-rs}h(s,L_{s}+\xi)\1_{[s\geq \nu]}ds\bigg\rbrace \\ & \times &  \max\limits_{\beta \in U} \bigg\lbrace  \mathbb{E}\bigg\lbrack \exp\bigg \lbrace \int_{\tau}^{\tau+\Delta}e^{-rs}h(s,L_{s}+\xi)\1_{[s\geq \nu]}ds -e^{-r(\tau+\Delta)}\psi(\beta) \bigg\rbrace \\ & \times &  Y_{\tau+\Delta}^{n}(\nu,\xi+\beta)|\Fc_{\tau}\bigg\rbrack \bigg\rbrace |\Fc_{t}\bigg\rbrack \\
 & \leq &  \esssup\limits_{\tau\in \mathcal{T}_{t}}\mathbb{E}\bigg\lbrack \exp\left\lbrace \frac{\gamma}{r}(e^{-rt}-e^{-r(\tau+\Delta)})  + \frac{\gamma e^{-r(\tau+\Delta)}}{r}\right\rbrace|\Fc_{t}\bigg\rbrack
  =   \exp(\frac{\gamma e^{-rt}}{r}).
\end{eqnarray*}
Therefore,
$$
\limsup_{t\rightarrow \infty}Y_{t}^{n+1}(\nu,\xi)\le 1.
$$
On the other hand
\begin{eqnarray*}
Y_{t}^{n+1}(\nu,\xi) &=& \esssup\limits_{\tau\in \mathcal{T}_{t}}\mathbb{E}\bigg\lbrack \exp\bigg\lbrace \int_{t}^{\tau}e^{-rs}h(s,L_{s}+\xi)\1_{[s\geq \nu]}ds\bigg\rbrace O_{\tau}^{n+1}(\nu,\xi)|\Fc_{t}\bigg\rbrack \\
& \geq & \lim\limits_{T\to + \infty} \mathbb{E}\bigg\lbrack \exp\bigg\lbrace \int_{t}^{T}e^{-rs}h(s,L_{s}+\xi)\1_{[s\geq \nu]}ds\bigg\rbrace O_{T}^{n+1}(\nu,\xi)|\Fc_{t}\bigg\rbrack \\
& \geq &  \mathbb{E}\left[\exp\left\lbrace \int_{t}^{+\infty}e^{-rs}h(s,L_{s}+\xi)\1_{[s\geq \nu]}ds\right\rbrace |\Fc_{t}\right]=Y_{t}^{0}(\nu,\xi).
\end{eqnarray*}since $\underset{T\to\infty}{\lim}O_{T}^{n+1}(\nu,\xi)=1$ by the induction hypothesis. Thus,
$$
\underset{t\to\infty}{\liminf}Y_{t}^{n+1}(\nu,\xi)\ge \lim_{t\rightarrow \infty}Y_{t}^{0}(\nu,\xi)= 1.
$$
This combined with the above estimates yield
$$
\lim_{t\rightarrow \infty}Y_{t}^{n+1}(\nu,\xi)=1.
$$
It remains to show that $Y^{n+1}(\nu,\xi)$ belongs to $\Sc_{c}^{2}$. With the above estimates, it is enough to show that it is continuous. First note that the process
$$ \Theta^{n+1}_{t}=  \exp \left\lbrace\int_{0}^{t}e^{-rs}h(s,L_{s}+\xi)\1_{[s\geq \nu]}ds\right\rbrace O_{t}^{n+1}(\nu,\xi), t\geq 0,$$is continuous on $[0,+\infty]$. Therefore, its Snell envelope is also continuous on $[0,+\infty]$, i.e., \\
$Y^{n+1}_t(\nu,\xi)\exp\{\int_0^th(s,L_s+\xi)1_{\{s\ge \nu\}}ds, \,\,t\ge 0$, is continuous
on $[0,+\infty]$ and then
$Y^{n+1}(\nu,\xi)$ is continuous on $[0,+\infty]$. The proof of the claim is now complete.

To show that $\P$-a.s. for every $t\ge 0$,
\begin{equation*}
Y_{t}^{n}(\nu, \xi) \leq Y_{t}^{n+1}(\nu, \xi),
\end{equation*}
it is enough to use an induction argument and to take into account that  $\P$-a.s., $\forall \xi \in \Fc_\nu$, $\forall t\ge 0$,
\begin{equation}\label{ineqy1y0}
Y_{t}^{1}(\nu, \xi) \geq Y_{t}^{0}(\nu, \xi).
\end{equation}
To see this last inequality holds, we note that, for any $T\geq t$,
\begin{equation*}
Y_{t}^{1}(\nu, \xi) \geq \mathbb{E}\left[\exp\left\lbrace \int_{t}^{T}e^{-rs}h(s,L_{s}+\xi)\1_{[s\geq \nu]}ds\right\rbrace O_{T}^{1}(\nu,\xi)|\Fc_{t}\right].
\end{equation*}
Take now the limit when $T\rightarrow \infty$ to obtain \eqref{ineqy1y0} since $
 \lim\limits_{T\to +\infty} O_{T}^{1}(\nu,\xi) =1.
 $

 Next, for $t\ge 0$ let us set $Y_{t}(\nu, \xi)=\lim_{n\rightarrow \infty}Y_{t}^{n}(\nu, \xi)$.
 Therefore,
 $Y_{t}(\nu, \xi)$ satisfies \eqref{estimyexp} by taking the limit in \eqref{ineqyn}. Now
 $(Y^{n}_t(\nu,\xi)\exp\{\int_0^th(s,L_s+\xi)1_{\{s\ge \nu\}}ds)_{t\ge 0}$ is a bounded increasing
 sequence of continuous supermartingales, then its limit is c{\`a}dl{\`a}g and
 then $Y(\nu, \xi)$ is c{\`a}dl{\`a}g. Finally by Part (II)-ii) of Appendix the process $O(\nu,\xi)$
 is càdlàg and the sequence $(O^n(\nu,\xi))_{n\ge 1}\nearrow O(\nu,\xi)$, therefore by Part (I)-(v) in Appendix, $Y(\nu, \xi)$ satisfies \eqref{MDynamic}.

 iii)  To show that for any two stopping times $\nu$ and $\nu^{\prime}$ such that $\nu\leq \nu^{\prime}$ and $\xi$ an $\Fc_{\nu}$-measurable r.v., we have  $\P $-a.s.
$$  Y_{t}(\nu, \xi)=Y_{t}(\nu^{\prime}, \xi), \quad \forall  \, t \geq \nu^{\prime} $$
it is enough to show that $\forall n\geq 0$, $\forall \xi\in \Fc_\nu$,
 $$  Y_{t}^n(\nu, \xi)=Y_{t}^n(\nu^{\prime}, \xi), \quad \forall t\ge \nu'.$$But this property is obtained by an induction. Actually for $n=0$ this property is valid in view of the definition of $Y_{t}^0(\nu, \xi)$ and since $1_{\{s\ge \nu\}}=1_{\{s\ge \nu'\}}$ if $s\ge t\ge \nu'\ge \nu$. Next assume that the property is valid for some $n$. Therefore, for any $\beta \in U$ (constant), by the induction hypothesis
 \begin{eqnarray*}\begin{array}{lll}
\mathbb{E}\bigg\lbrack \exp\bigg\lbrace \int_{t}^{t+\Delta}e^{-rs}h(s,L_{s}+\xi)\1_{[s\geq \nu]}ds -e^{-r(t+\Delta)}\psi(\beta) \bigg\rbrace  \times  Y_{t+\Delta}^{n}(\nu,\xi+\beta)|\Fc_{t}\bigg\rbrack \\ \qquad\qquad \quad=
 \mathbb{E}\bigg\lbrack \exp\bigg\lbrace \int_{t}^{t+\Delta}e^{-rs}h(s,L_{s}+\xi)\1_{[s\geq \nu']}ds -e^{-r(t+\Delta)}\psi(\beta) \bigg\rbrace  \times  Y_{t+\Delta}^{n}(\nu',\xi+\beta)|\Fc_{t}\bigg\rbrack.
 \end{array}
\end{eqnarray*}
 Taking the supremum over $\beta \in U$, we obtain $O_{t}^{n+1}(\nu, \xi)=O_{t}^{n+1}(\nu', \xi)$, and then $Y_{t}^{n+1}(\nu, \xi)=Y_{t}^{n+1}(\nu^{\prime}, \xi)$. To complete the proof, we just need to take the limit w.r.t. $n$.
 \end{proof}

 \begin{Lemma} For any stopping time $\nu$ and $\xi$ a finite r.v.
 (i.e. $card(\xi(\Omega))<\infty$), $\Fc_\nu$-measurable we have:
\begin{equation}\label {ydetermine}
\forall t\ge \nu,\quad Y_t(\nu,\xi)=\sum_{\theta \in
\xi(\Omega)}1_{\{\xi=\theta\}}Y_t(\nu,\theta).
\end{equation}
\end{Lemma}
\proof It is enough to show that for any $n\ge 0$, for any $\xi \in \Fc_\nu$
finite \begin{equation}\label {yndetermine}
\forall t\ge \nu,\quad Y^n_t(\nu,\xi)=\sum_{\theta \in
\xi(\Omega)}1_{\{\xi=\theta\}}Y^n_t(\nu,\theta).
\end{equation}
This last equality will be shown by induction. Indeed, for $n=0$ the property holds true since
\begin{align}
\nn Y_{t}^{0}(\nu,\xi )&= \mathbb{E}\left[\exp\left\lbrace \int_{t}^{+\infty}e^{-rs}h(s,L_{s}+\xi)\1_{[s\geq \nu]}ds\right\rbrace|\Fc_{t}\right] \\
 &\nn =\mathbb{E}\left[\sum_{\theta \in
\xi(\Omega)}1_{\{\xi=\theta\}}\exp\left\lbrace \int_{t}^{+\infty}e^{-rs}h(s,L_{s}+\theta)\1_{[s\geq \nu]}ds\right\rbrace|\Fc_{t}\right]\\\nn
&\nn =\sum_{\theta \in
\xi(\Omega)}1_{\{\xi=\theta\}}\mathbb{E}\left[\exp\left\lbrace \int_{t}^{+\infty}e^{-rs}h(s,L_{s}+\theta)\1_{[s\geq \nu]}ds\right\rbrace|\Fc_{t}\right]=\sum_{\theta \in
\xi(\Omega)}1_{\{\xi=\theta\}}Y^0_t(\nu,\theta)
\end{align}
since $\{\xi=\theta\}\in \Fc_\nu\subset \Fc_t$. \\
Suppose now that the property holds for some $n\ge 0$. Let us show
that it holds also for $n+1$. For that let us set, for $t\ge \nu$,
\begin{align} \nn \bar Y^{n+1}_t(\nu,\xi)=\sum_{\theta \in
\xi(\Omega)}1_{\{\xi=\theta\}}Y^{n+1}_t(\nu,\theta).\end{align}
First note that, for any $t\ge \nu$,
$$
\exp\{\int_{\nu}^{t}e^{-rs}h(s,L_{s}+\xi)\1_{[s\geq
\nu]}ds\}\times \bar Y^{n+1}_t(\nu,\xi)=\sum_{\theta \in
\xi(\Omega)}1_{\{\xi=\theta\}}\exp\{\int_{\nu}^{t}e^{-rs}h(s,L_{s}+\theta)\1_{[s\geq
\nu]}ds\}\times Y^{n+1}_t(\nu,\theta).
$$
Therefore,
$$\left(\exp\{\int_{\nu}^{t}e^{-rs}h(s,L_{s}+\xi)\1_{[s\geq \nu]}ds\}\times \bar Y^{n+1}_t(\nu,\xi)\right)_{t\ge \nu}$$ is a continuous supermartingale since $\exp\{\int_{\nu}^{t}e^{-rs}h(s,L_{s}+\theta)\1_{[s\geq \nu]}ds\}\times Y^{n+1}_t(\nu,\theta)$, $t\ge \nu$, are continuous supermartingales  and the sets $\{\xi=\theta\}$ belong to $\Fc_\nu$. On the other hand for any $ t\ge\nu$,
\begin{align}
&\nn \sum_{\theta \in
\xi(\Omega)}1_{\{\xi=\theta\}}\exp\{\int_{\nu}^{t}e^{-rs}h(s,L_{s}+\theta)\1_{[s\geq \nu]}ds\}\times Y^{n+1}_t(\nu,\theta)\\&\qquad \ge
\sum_{\theta \in
\xi(\Omega)}1_{\{\xi=\theta\}} \exp\{\int_{\nu}^{t}e^{-rs}h(s,L_{s}+\theta)\1_{[s\geq \nu]}ds\}\times O^{n+1}_t(\nu,\theta)=:V_t\nn
\end{align}
since $(\exp\{\int_{\nu}^{t}e^{-rs}h(s,L_{s}+\theta)\1_{[s\geq \nu]}ds\}\times Y^{n+1}_t(\nu,\theta))_{t\ge \nu}$ is the Snell envelope of
$(\exp\{\int_{\nu}^{t}e^{-rs}h(s,L_{s}+\theta)\1_{[s\geq \nu]}ds\}\times O^{n+1}_t(\nu,\theta))_{t\ge \nu}$ for any $\theta \in \xi(\Omega)$. But, by using the induction hypothesis (in the penultimate equality), we have
\begin{align}V_t&=\sum_{\theta \in
\xi(\Omega)}1_{\{\xi=\theta\}} \exp\{\int_{\nu}^{t}e^{-rs}h(s,L_{s}+\theta)\1_{[s\geq \nu]}ds\}\times O^{n+1}_t(\nu,\theta)\nn\\&=
\sum_{\theta \in
\xi(\Omega)}1_{\{\xi=\theta\}} \max\limits_{\beta \in U} \bigg\lbrace  \mathbb{E}\bigg\lbrack \exp\bigg\lbrace \int_{\nu}^{t+\Delta}e^{-rs}h(s,L_{s}+\theta)\1_{[s\geq \nu]}ds -e^{-r(t+\Delta)}\psi(\beta) \bigg\rbrace  \times  Y_{t+\Delta}^{n}(\nu,\theta+\beta)|\Fc_{t}\bigg\rbrack \bigg\rbrace\nn\\&
= \max\limits_{\beta \in U} \bigg\lbrace \sum_{\theta \in
\xi(\Omega)} 1_{\{\xi=\theta\}}\mathbb{E}\bigg\lbrack \exp\bigg\lbrace \int_{\nu}^{t+\Delta}e^{-rs}h(s,L_{s}+\theta)\1_{[s\geq \nu]}ds -e^{-r(t+\Delta)}\psi(\beta) \bigg\rbrace  \times  Y_{t+\Delta}^{n}(\nu,\theta+\beta)|\Fc_{t}\bigg\rbrack \bigg\rbrace\nn
\\&
= \max\limits_{\beta \in U} \bigg\lbrace \mathbb{E}\bigg\lbrack \sum_{\theta \in
\xi(\Omega)} 1_{\{\xi=\theta\}}\exp\bigg\lbrace \int_{\nu}^{t+\Delta}e^{-rs}h(s,L_{s}+\theta)\1_{[s\geq \nu]}ds -e^{-r(t+\Delta)}\psi(\beta) \bigg\rbrace  \times  Y_{t+\Delta}^{n}(\nu,\theta+\beta)|\Fc_{t}\bigg\rbrack \bigg\rbrace\nn\\
&
= \max\limits_{\beta \in U} \bigg\lbrace \mathbb{E}\bigg\lbrack \sum_{\theta \in
\xi(\Omega)} 1_{\{\xi=\theta\}}\exp\bigg\lbrace \int_{\nu}^{t+\Delta}e^{-rs}h(s,L_{s}+\xi)\1_{[s\geq \nu]}ds -e^{-r(t+\Delta)}\psi(\beta) \bigg\rbrace  \times  Y_{t+\Delta}^{n}(\nu,\theta+\beta)|\Fc_{t}\bigg\rbrack \bigg\rbrace\nn
\\
&
= \max\limits_{\beta \in U} \bigg\lbrace \mathbb{E}\bigg\lbrack \exp\bigg\lbrace \int_{\nu}^{t+\Delta}e^{-rs}h(s,L_{s}+\xi)\1_{[s\geq \nu]}ds -e^{-r(t+\Delta)}\psi(\beta) \bigg\rbrace\{  \sum_{\theta \in
\xi(\Omega)} 1_{\{\xi=\theta\}} Y_{t+\Delta}^{n}(\nu,\theta+\beta)\}|\Fc_{t}\bigg\rbrack \bigg\rbrace\nn
\\
&
= \max\limits_{\beta \in U} \bigg\lbrace \mathbb{E}\bigg\lbrack \exp\bigg\lbrace \int_{\nu}^{t+\Delta}e^{-rs}h(s,L_{s}+\xi)\1_{[s\geq \nu]}ds -e^{-r(t+\Delta)}\psi(\beta) \bigg\rbrace Y_{t+\Delta}^{n}(\nu,\xi+\beta)|\Fc_{t}\bigg\rbrack \bigg\rbrace\nn\\&\nn
=\exp\{\int_{\nu}^{t}e^{-rs}h(s,L_{s}+\xi)\1_{[s\geq \nu]}ds\}\times O^{n+1}_t(\nu,\xi)\nn.
\end{align}
Thus, the continuous supermartingale
$(\exp\{\int_{\nu}^{t}e^{-rs}h(s,L_{s}+\theta)\1_{[s\geq
\nu]}ds\}\times \bar Y^{n+1}_t(\nu,\xi))_{t\ge \nu}$ is greater than
the process $(\exp\{\int_{\nu}^{t}e^{-rs}h(s,L_{s}+\xi)\1_{[s\geq
\nu]}ds\}\times O^{n+1}_t(\nu,\xi))_{t\ge \nu}$. Next, let
$(U_t)_{t\ge \nu}$ be a $\cdlg$ supermartinagle such that, for every $t\ge
\nu$,
\begin{align} \nn U_t&\ge
\exp\{\int_{\nu}^{t}e^{-rs}h(s,L_{s}+\xi)\1_{[s\geq \nu]}ds\}\times
O^{n+1}_t(\nu,\xi)\\\nn &=\sum_{\theta \in
\xi(\Omega)}1_{\{\xi=\theta\}}
\exp\{\int_{\nu}^{t}e^{-rs}h(s,L_{s}+\theta)\1_{[s\geq
\nu]}ds\}\times O^{n+1}_t(\nu,\theta).\end{align} This implies that,
for any $\theta \in \xi(\Omega)$ and $t\ge \nu$,
\begin{align}
1_{\{\xi=\theta\}} U_t\ge 1_{\{\xi=\theta\}}
\exp\{\int_{\nu}^{t}e^{-rs}h(s,L_{s}+\theta)\1_{[s\geq
\nu]}ds\}\times O^{n+1}_t(\nu,\theta).\nn \end{align} But, since the
set $\{\xi=\theta\}$ belongs to $\Fc_\nu$,  the process
$(1_{\{\xi=\theta\}}\exp\{\int_{\nu}^{t}e^{-rs}h(s,L_{s}+\theta)\1_{[s\geq
\nu]}ds\}\times Y^{n+1}_t(\nu,\theta))_{t\ge \nu}$ is the Snell
envelope of $( 1_{\{\xi=\theta\}}
\exp\{\int_{\nu}^{t}e^{-rs}h(s,L_{s}+\theta)\1_{[s\geq
\nu]}ds\}\times O^{n+1}_t(\nu,\theta))_{t\ge \nu}. $ Now, as
$(1_{\{\xi=\theta\}} U_t)_{t \ge \nu}$ is still a $\cdlg$ supermartingale then by Part (I), we have, for any $t\ge \nu$,
$$
1_{\{\xi=\theta\}} U_t\ge 1_{\{\xi=\theta\}}\exp\{\int_{\nu}^{t}e^{-rs}h(s,L_{s}+\theta)\1_{[s\geq \nu]}ds\}\times Y^{n+1}_t(\nu,\theta).$$ 
This implies that, for any $t\ge \nu$,
$$
U_t=\sum_{\theta \in
\xi(\Omega)}1_{\{\xi=\theta\}} U_t\ge \sum_{\theta \in
\xi(\Omega)}1_{\{\xi=\theta\}}\exp\{\int_{\nu}^{t}e^{-rs}h(s,L_{s}+\theta)\1_{[s\geq \nu]}ds\}\times Y^{n+1}_t(\nu,\theta).
$$
Consequently, the process $(\sum_{\theta \in
\xi(\Omega)}1_{\{\xi=\theta\}}\exp\{\int_{\nu}^{t}e^{-rs}h(s,L_{s}+\theta)\1_{[s\geq \nu]}ds\}\times Y^{n+1}_t(\nu,\theta))_{t\ge \nu}$ is the smallest $\cdlg$ supermartingale which dominates
$(\exp\{\int_{\nu}^{t}e^{-rs}h(s,L_{s}+\xi)\1_{[s\geq \nu]}ds\}\times O^{n+1}_t(\nu,\xi))_{t\ge \nu}$, and then, it is its  Snell envelope, i.e., for any $t\ge \nu$,
\begin{align} \nn
&\exp\{\int_{\nu}^{t}e^{-rs}h(s,L_{s}+\xi)\1_{[s\geq \nu]}ds\}\times
Y^{n+1}_t(\nu,\xi)\\ \nn &=\sum_{\theta \in
\xi(\Omega)}1_{\{\xi=\theta\}}\exp\{\int_{\nu}^{t}e^{-rs}h(s,L_{s}+\theta)\1_{[s\geq
\nu]}ds\}\times Y^{n+1}_t(\nu,\theta)\\\nn
&=\exp\{\int_{\nu}^{t}e^{-rs}h(s,L_{s}+\xi)\1_{[s\geq \nu]}ds\times
\sum_{\theta \in \xi(\Omega)}1_{\{\xi=\theta\}}Y^{n+1}_t(\nu,\theta)
\end{align}which implies \eqref{yndetermine}
holds for $n+1$ after an obvious simplification. It follows that for
any $n\ge 0$, the property \eqref{yndetermine} holds. Now it is
enough to take the limit w.r.t $n$ in \eqref{yndetermine} to obtain
the claim \eqref{ydetermine}.
\begin{Remark}As in Proposition \ref{53}, we can show in the same way that for any $n\ge 0$, there exists a strategy
$\delta^{*}_{n}$ which belongs to $\Ac_{n}$ such that
\begin{equation*}
Y_{0}^{n}(0,0)=\sup \limits_{\delta \in \Ac_{n}}J(\delta)=J(\delta^{*}_{n}),
\end{equation*} i.e., $\delta^{*}_{n}$ is optimal in $\Ac_{n}$.\end{Remark}
\subsection{The optimal strategy for the risk-sensitive problem }
We now deal with the issue of existence of an optimal strategy for
the risk-sensitive impulse control problem with delay. The main
difficulty is related to continuity of the process $Y(\nu,\xi)$.
Once this property is established we exhibit an optimal strategy and
show that $Y(0,0)$ is the value function of the control problem.  We have
\begin{Proposition} Let Assumption \ref{assumpt} hold. Then the process $(Y_t(\nu,\xi))_{t\ge 0}$ defined in \eqref{MDynamic} is continuous.
\end{Proposition}
\begin{proof}
The proof is similar to the one of Proposition \ref{Proposition_cont}. First let us notice that the process $(O_{t}(\nu,\xi))_{t \geq 0}$ is c\`adl\`ag since $Y(\nu,\xi)$ is c\`adl\`ag (see Appendix Part (II)). Next, let $T$ be a predictable stopping time such that $\Delta_T Y(\nu,\xi)<0$. This implies that the process
$(O_t(\nu,\xi))_{t\geq 0}$ has a negative jump at $T$ and $O_{T^{-}}(\nu,\xi)=Y_{T^{-}}(\nu,\xi)$ (see Appendix, Part (I)).
 Therefore,
\begin{eqnarray*}
&{}&O_{T^{-}}(\nu,\xi)- O_{T}(\nu,\xi)\\
&= & \max_{\beta \in U} \bigg\lbrace \E \bigg\lbrack \exp \bigg\lbrace \int_{T}^{T+\Delta}e^{-rs}h(s,L_{s}+\xi)\1_{[s\geq \nu]}ds -e^{-r(T+\Delta)}\psi(\beta)\bigg\rbrace \times Y_{(T+\Delta)^{-}}(\nu,\xi+\beta)|\Fc_{T}\bigg\rbrack \bigg\rbrace  \\& {} &\quad -  \max_{\beta \in U} \bigg\lbrace \E \bigg\lbrack \exp \bigg\lbrace \int_{T}^{T+\Delta}e^{-rs}h(s,L_{s}+\xi)\1_{[s\geq \nu]}ds - e^{-r(T+\Delta)}\psi(\beta)\bigg\rbrace Y_{T+\Delta}(\nu,\xi+\beta)|\Fc_{T}\bigg\rbrack \bigg\rbrace \\
&\leq &
\max_{\beta \in U} \bigg\lbrace \E \bigg\lbrack \exp \bigg\lbrace \int_{T}^{T+\Delta}e^{-rs}h(s,L_{s}+\xi)\1_{[s\geq \nu]}ds -e^{-r(T+\Delta)}\psi(\beta)\bigg\rbrace \\ &{} &\qquad \qquad \times \bigg( Y_{(T+\Delta)^{-}}(\nu,\xi+\beta)-Y_{T+\Delta}(\nu,\xi+\beta)\bigg)|\Fc_{T}\bigg\rbrack \bigg\rbrace \\
 &=&
 \max_{\beta \in U} \bigg\lbrace \E \bigg\lbrack  \1_{A_{T+\Delta}(\xi+\beta)} \exp \bigg\lbrace \int_{T}^{T+\Delta}e^{-rs}h(s,L_{s}+\xi)\1_{[s\geq \nu]}ds -e^{-r(T+\Delta)}\psi(\beta)\bigg\rbrace \\ & {} &\qquad\qquad\times  \bigg( Y_{(T+\Delta)^{-}}(\nu,\xi+\beta)-Y_{T+\Delta}(\nu,\xi+\beta)\bigg)|\Fc_{T}\bigg\rbrack \bigg\rbrace \\
 &\leq &
 \max_{\beta \in U} \bigg\lbrace \E \bigg\lbrack  \1_{A_{T+\Delta}(\xi+\beta)} \exp \bigg\lbrace \int_{T}^{T+\Delta}e^{-rs} \gamma ds -ke^{-r(T+\Delta)} ds \bigg\rbrace \\ & {} &\qquad\qquad\times \bigg( Y_{(T+\Delta)^{-}}(\nu,\xi+\beta)-Y_{T+\Delta}(\nu,\xi+\beta)\bigg)|\Fc_{T}\bigg\rbrack \bigg\rbrace \\
  &\leq &
 \max_{\beta \in U} \bigg\lbrace \E \bigg\lbrack  \1_{A_{T+\Delta}(\xi+\beta)} \exp \bigg\lbrace \frac{\gamma}{r}(e^{-rT}-e^{-r(T+\Delta)}) \bigg\rbrace \\ & {} &\qquad\qquad\times \bigg( Y_{(T+\Delta)^{-}}(\nu,\xi+\beta)-Y_{T+\Delta}(\nu,\xi+\beta)\bigg)|\Fc_{T}\bigg\rbrack \bigg\rbrace \\
\end{eqnarray*}
where for any predictable stopping time $T\geq \nu$ and $\xi$ an $\Fc_\nu$-measurable r.v.\\ $A_T(\xi):=\{\omega \in \Omega, \Delta_TY(\nu,\xi)<0\}$ which belongs to $\Fc_T$. Therefore,
\begin{eqnarray}\label{cont1} \nonumber
\1_{A_{T}(\xi)}\{O_{T^{-}}(\nu,\xi)- O_{T}(\nu,\xi)\} &\leq &
 \max_{\beta \in U} \bigg\lbrace \E \bigg\lbrack \1_{A_{T}(\xi)} \times \1_{A_{T+\Delta}(\xi+\beta)} \exp \bigg\lbrace \frac{\gamma}{r}(e^{-rT}-e^{-r(T+\Delta)}) \bigg\rbrace \\ &\times & \bigg( Y_{(T+\Delta)^{-}}(\nu,\xi+\beta)-Y_{T+\Delta}(\nu,\xi+\beta)\bigg)|\Fc_{T}\bigg\rbrack \bigg\rbrace.
\end{eqnarray}
We note that there exists at least one $\beta \in U$ such that
the right-hand side is positive. Otherwise the left-hand side is
null and this is a contradiction. Since 
$Y_{T+\Delta}(\nu,\xi+\beta)\geq O_{T+\Delta}(\nu,\xi+\beta)$ and on
the set $A_{T+\Delta}(\xi+\beta)$,\\
$Y_{(T+\Delta)^{-}}(\nu,\xi+\beta)=O_{(T+\Delta)^{-}}(\nu,\xi+\beta)$.
Therefore, \eqref{cont1} implies
\begin{eqnarray*}
&{}&\1_{A_{T}(\xi)}\{O_{T^{-}}(\nu,\xi)-O_{T}(\nu,\xi)\}  \\ &{}&\le
\max_{\beta \in U} \bigg\lbrace \E \bigg\lbrack \1_{A_{T}(\xi)} \times \1_{A_{T+\Delta}(\xi+\beta)} \exp \bigg\lbrace \frac{\gamma}{r}(e^{-rT}-e^{-r(T+\Delta)}) \bigg\rbrace \\ &{}& \qquad \times  \bigg( O_{(T+\Delta)^{-}}(\nu,\xi+\beta)-O_{T+\Delta}(\nu,\xi+\beta)\bigg)|\Fc_{T}\bigg\rbrack \bigg\rbrace \\
 &{}&  \leq  \E \bigg\lbrack \1_{A_{T}(\xi)}\times \exp \bigg\lbrace \frac{\gamma}{r}(e^{-rT}-e^{-r(T+\Delta)}) \bigg\rbrace \\ &{}& \qquad \times  \max_{\beta \in U}\bigg\lbrace \E \bigg\lbrack \1_{A_{T+\Delta}(\xi+\beta)} \bigg\lbrace O_{(T+\Delta)^{-}}(\nu,\xi+\beta) - O_{T+\Delta}(\nu,\xi+\beta) \bigg\rbrace|\Fc_{T+\Delta} \bigg\rbrack \bigg\rbrace |\Fc_{T} \bigg\rbrack \\
&{}&  \leq  \E \bigg\lbrack \1_{A_{T}(\xi)} \exp \bigg\lbrace \frac{\gamma}{r}(e^{-rT}-e^{-r(T+\Delta)})\bigg\rbrace \times  \E \bigg\lbrack \1_{A_{T+\Delta}(\xi+\beta_{1})} \bigg\lbrace O_{(T+\Delta)^{-}}(\nu,\xi+\beta_{1}) \\ &{}& \qquad -O_{T+\Delta}(\nu,\xi+\beta_{1}) \bigg\rbrace|\Fc_{T+\Delta} \bigg\rbrack |\Fc_{T} \bigg\rbrack,
\end{eqnarray*}
where $\beta_1$ is a r.v. $\Fc_{T+\Delta}$-measurable valued in $U$.
The construction of the r.v. $\beta_1$ is similar as the one in the
proof of Proposition \ref{Proposition_cont} (see \eqref{inside3}) by using the property
\eqref{ydetermine}. Note that, as previously, the left-hand side is
not null. Next, since we have that  $A_{T}(\xi)$ and $\left( \exp
\bigg\lbrace \frac{\gamma}{r}(e^{-rT}-e^{-r(T+\Delta)}) \bigg\rbrace
\right) $ are  also $\Fc_{T+\Delta}$-measurable then
\begin{eqnarray} \nonumber \label{cont2}
\1_{A_{T}(\xi)}\{O_{T^{-}}(\nu,\xi)- O_{T}(\nu,\xi)\} &\leq &
 \E \bigg\lbrack \1_{A_{T}(\xi)}\times \1_{A_{T+\Delta}(\xi+\beta_{1})} \exp \bigg\lbrace \frac{\gamma}{r}(e^{-rT}-e^{-r(T+\Delta)}) \bigg\rbrace \\ &\times & \bigg\lbrace O_{(T+\Delta)^{-}}(\nu,\xi+\beta_{1})-O_{T+\Delta}(\nu,\xi+\beta_{1})
\bigg\rbrace |\Fc_{T} \bigg\rbrack .
\end{eqnarray}
Now by repeating this reasoning one deduces the existence of a sequence of $U$-valued random variables
$(\beta_k)_{k\ge 1}$ such that $\beta_k$ is $\Fc_{T+k\Delta}$-measurable and for any $n\ge 1$,
\begin{equation}\label{cont3}
\begin{array}{l}
\1_{A_{T}(\xi)}\{O_{T^{-}}(\nu,\xi)- O_{T}(\nu,\xi)\} \\
\qquad \leq
 \E \bigg\lbrack \1_{A_{T}(\xi)}\bigg\lbrace \prod_{k=1}^{k=n} 1_{A_{T+k\Delta}(\xi+\beta_1+ \cdots +\beta_k)} \exp \bigg\lbrace \frac{\gamma}{r}(e^{-rT}-e^{-r(T+n\Delta)}) \bigg\rbrace \\ \qquad \qquad \times \bigg( O_{(T+n\Delta)^{-}}(\nu,\xi+\beta_1+ \cdots +\beta_n)-O_{T+n\Delta}(\nu,\xi+\beta_1+ \cdots +\beta_n)\bigg) \bigg\rbrace |\Fc_{T}\bigg\rbrack.
\end{array}
\end{equation}
But, \begin{equation}
 |O_t(\nu,\xi)|\leq \exp \left\lbrace \frac{\gamma}{r}\{e^{-rt}-e^{-r(t+\Delta)}\}
+\|\psi\|e^{-r(t+\Delta)}+\frac{\gamma}{r} e^{-r(t+\Delta)}\right\rbrace,
\end{equation}then, setting $\Sigma_n=\beta_1+ \cdots +\beta_n$ $(n\ge 1)$, one obviously has
$$\limsup_{n\rightarrow \infty}O_{(T+n\Delta)^{-}}(\nu,\xi+\Sigma_n)\le 1.$$
On the other hand there exists a subsequence $(n_k)_{k\ge 1}$ such that
\begin{equation}\label{tendre1}\P\mbox{-}a.s.\,\quad
\underset{k\to\infty}{\lim}O_{T+n_k\Delta}(\nu,\xi+\Sigma_{n_k})=1.
\end{equation}
Indeed, by construction and \eqref{estimyexp}, for any $\beta \in U$, $\P-a.s.$,
$$Y^0_{T+n\Delta}(\nu, \xi+\Sigma_n+\beta)\le  Y_{T+n\Delta}(\nu, \xi+\Sigma_n+\beta)\leq \exp(\frac{\gamma e^{-r(T+n\Delta)}}{r}).$$As $\underset{n\to\infty}{\lim}Y^0_{T+n\Delta}(\nu, \xi+\Sigma_n+\beta)=1$, then \begin{equation}\label{tendre2}\lim_{n\rightarrow \infty}Y_{T+n\Delta}(\nu, \xi+\Sigma_n+\beta)=1.\end{equation}
Next recall the definition of the process $O(\nu,\xi)$ to obtain
that
\begin{align}
&\nonumber \E[|O_{T+n\Delta}(\nu,\xi+\Sigma_{n})-1|]\\\nonumber&=\E\bigg\lbrack|\max_{\beta \in U}\bigg\lbrace\E \bigg\lbrack \bigg\lbrace \exp \bigg\lbrace \int_{T}^{T+n\Delta}e^{-rs}h(s,L_{s}+\xi+\Sigma_{n})\1_{[s\geq \nu]}ds \\\nonumber&- e^{-r(T+n\Delta)}\psi(\beta)\bigg\rbrace Y_{T+n\Delta}(\nu,\xi+\Sigma_{n}+\beta)-1\bigg\rbrace|\Fc_{T+n\Delta}\bigg\rbrack |\bigg\rbrack\\
&\nonumber\le \sum_{\beta \in U}\E\bigg\lbrack | \exp \bigg\lbrace \int_{T}^{T+n\Delta}e^{-rs}h(s,L_{s}+\xi+\Sigma_{n})\1_{[s\geq \nu]}ds \\\nonumber&\qquad\qquad - e^{-r(T+n\Delta)}\psi(\beta)\bigg\rbrace Y_{T+n\Delta}(\nu,\xi+\Sigma_{n}+\beta)-1|\bigg\rbrack  .
\end{align}
But, by the Lebesgue Theorem and \eqref{tendre2}, the last term converges to 0 as $n\rightarrow \infty$, therefore one can substract a subsequence $(n_k)_{k\ge 1}$ such that \eqref{tendre1} holds.

Let us now consider this subsequence which we still denote by $\{n\}$ and go back now to \eqref{cont3}. By using the conditional Fatou's Lemma we obtain
\begin{equation}\label{cont4}
\begin{array}{l}
\1_{A_{T}(\xi)}\{O_{T^{-}}(\nu,\xi)- O_{T}(\nu,\xi)\} \\
\qquad \leq \underset{n\to\infty}{\limsup} \E \bigg\lbrack \1_{A_{T}(\xi)}\bigg\lbrace \prod_{k=1}^{k=n} 1_{A_{T+k\Delta}(\xi+\beta_1+ \cdots +\beta_k)} \exp \bigg\lbrace \frac{\gamma}{r}(e^{-rT}-e^{-r(T+n\Delta)}) \bigg\rbrace \\ \qquad \qquad \times \bigg( O_{(T+n\Delta)^{-}}(\nu,\xi+\beta_1+ \cdots +\beta_n)-O_{T+n\Delta}(\nu,\xi+\beta_1+ \cdots +\beta_n)\bigg) \bigg\rbrace |\Fc_{T}\bigg\rbrack\\
\qquad \leq
 \E \bigg\lbrack \underset{n\to\infty}{\limsup}\1_{A_{T}(\xi)}\bigg\lbrace \prod_{k=1}^{k=n} 1_{A_{T+k\Delta}(\xi+\beta_1+ \cdots +\beta_k)} \exp \bigg\lbrace \frac{\gamma}{r}(e^{-rT}-e^{-r(T+n\Delta)}) \bigg\rbrace \\ \qquad \qquad \times \bigg( O_{(T+n\Delta)^{-}}(\nu,\xi+\beta_1+ \cdots +\beta_n)-O_{T+n\Delta}(\nu,\xi+\beta_1+ \cdots +\beta_n)\bigg) \bigg\rbrace |\Fc_{T}\bigg\rbrack\\
\qquad \leq
 \E \bigg\lbrack e^{-rT}\lbrace\underset{n\to\infty}{\limsup} \,O_{(T+n\Delta)^{-}}(\nu,\xi+\beta_1+ \cdots +\beta_n)-\\\qquad\qquad\qquad\qquad\underset{n\to\infty}{\lim}O_{T+n\Delta}(\nu,\xi+\beta_1+ \cdots +\beta_n) |\Fc_{T}\bigg\rbrack \le 0.
\end{array}
\end{equation}
This in turn  implies that
$$
\1_{A_{T}(\xi)}\{O_{T^{-}}(\nu,\xi)-O_{T}(\nu,\xi)\}=0,
$$
which leads to  a contradiction. Therefore, the process $Y(\nu,\xi)$ is continuous.
\end{proof}
We are now ready to give the main result of this section.
\begin{Theorem} Assume that \ref{assumpt} hold. Let us define  the  strategy $\delta^{*}=(\tau_{n}^{*},\beta_{n}^{*})_{n\geq 0}$  by
$$\tau_{0}^{*}= \begin{cases}\inf\{ s\in[0,\infty), O_{s}(0,0)\geq Y_{s}(0,0)\},\\  + \infty \quad \text{otherwise} \end{cases}$$ \\and $\beta_{0}^{*}$ is an $\mathcal{F}_{\tau_{0}^{*}}$-r.v. valued in $U$ such that
$$ O_{\tau_{0}^{*}}(0,0):= \E\left[\exp \bigg\lbrace \Int_{\tau_{0}^{*}}^{\tau_{0}^{*}+\Delta}  e^{-rs}  h(s,L_{s})ds- e^{-r(\tau_{0}^{*}+\Delta)}\psi(\beta_{0}^{*})\bigg\rbrace Y_{\tau_{0}^{*}+\Delta}(\tau_{0}^{*},\beta_{0}^{*})|\mathcal{F}_{\tau_{0}^{*}}\right].$$For $n \geq 1$,
 $$ \tau_{n}^{*}=\inf\bigg\lbrace s\geq \tau_{n-1}^{*}+\Delta , O_{s}(\tau_{n-1}^{*},\beta_{0}^{*}+\cdots +\beta_{n-1}^{*} )\geq Y_{s}(\tau_{n-1}^{*},\beta_{0}^{*}+\cdots +\beta_{n-1}^{*} )\bigg\rbrace,$$
and $\beta_{n}^{*}$ is an $\mathcal{F}_{\tau_{n}^{*}}$-r.v. valued in $U$ such that
 \begin{eqnarray*}
 O_{\tau_{n}^{*}}(\tau_{n-1}^{*},\beta_{0}^{*}+\cdots +\beta_{n-1}^{*} )&=&\E \bigg\lbrack \exp\bigg\lbrace \Int_{\tau_{n}^{*}}^{\tau_{n}^{*}+\Delta}  e^{-rs}  h(s,L_{s}+\beta_{0}^{*}+\cdots +\beta_{n-1}^{*})ds \\ &-& e^{-r(\tau_{n}^{*}+\Delta)}\psi(\beta_{n}^{*})\bigg\rbrace Y_{\tau_{n}^{*}+\Delta}(\tau_{n}^{*},\beta_{0}^{*}+\cdots +\beta_{n-1}^{*}+\beta_{n}^{*} )|\mathcal{F}_{\tau_{n}^{*}}\bigg\rbrack.
 \end{eqnarray*}
 Then, the strategy $\delta^{*}=(\tau_{n}^{*},\beta_{n}^{*})_{n\geq 0}$ is optimal for the risk-sensitive impulse control problem, i.e.,$$ Y_{0}(0,0)=\Sup_{\delta \in \mathcal{A}}J(\delta)=J(\delta^{*}).$$
\end{Theorem}
\begin{proof} First let us make precise the way the r.v. $\beta_{n}^{*}$ is
constructed. For $i=1,\ldots,p$, let $\A_i$ be the  set
\begin{align}
 \nn &\A_i:=\bigg\lbrace
   \max\limits_{\beta\in U}\E \bigg\lbrack \exp\bigg\lbrace
- e^{-r(\tau_{n}^{*}+\Delta)} \psi(\beta)\bigg\rbrace \times
Y_{\tau_{n}^{*}+\Delta}(\tau_{n}^{*}, \beta_{0}^{*}+\cdots
+\beta_{n-1}^{*}+\beta)|
\mathcal{F}_{\tau_{n}^{*}}\bigg\rbrack\\&\nn \qquad\qquad\qquad = \E
\bigg\lbrack \exp\bigg\lbrace - e^{-r(\tau_{n}^{*}+\Delta)}
\psi(\beta_i)\bigg\rbrace \times
Y_{\tau_{n}^{*}+\Delta}(\tau_{n}^{*}, \beta_{0}^{*}+\cdots
+\beta_{n-1}^{*}+\beta_i)| \mathcal{F}_{\tau_{n}^{*}}\bigg\rbrack
\bigg\rbrace.
\end{align}
We define $\beta_{n}^{*}$ as
$$
\beta_{n}^{*}=\beta_1 \mbox{ on } \A_1 \text{ and  }
\beta_{n}^{*}=\beta_j \mbox{ on
}\A_j\backslash\bigcup_{k=1}^{j-1}\A_k \text{ for }j=2,\ldots, p.
$$
Thus, using \eqref{ydetermine} and as in \eqref{inside1},
$\beta_{n}^{*}$ satisfies
\begin{align}&\nn \E
\bigg\lbrack \exp\bigg\lbrace - e^{-r(\tau_{n}^{*}+\Delta)}
\psi(\beta_n^*)\bigg\rbrace \times
Y_{\tau_{n}^{*}+\Delta}(\tau_{n}^{*}, \beta_{0}^{*}+\cdots
+\beta_{n-1}^{*}+\beta_n^*)| \mathcal{F}_{\tau_{n}^{*}}\bigg\rbrack
\\& =\max\limits_{\beta\in U}\E \bigg\lbrack \exp\bigg\lbrace
- e^{-r(\tau_{n}^{*}+\Delta)} \psi(\beta)\bigg\rbrace \times
Y_{\tau_{n}^{*}+\Delta}(\tau_{n}^{*}, \beta_{0}^{*}+\cdots
+\beta_{n-1}^{*}+\beta)| \mathcal{F}_{\tau_{n}^{*}}\bigg\rbrack .
\end{align}
 
 The proof of the theorem is performed in the following steps.
 
 \medskip
\noindent  \underline{Step 1}: $Y_{0}(0,0)= J(\delta^{*}) $. \\

\noindent  We have 
\begin{equation}
 Y_{0}(0,0)= \esssup\limits_{\tau \in
\mathcal{T}_0}\E\left[\exp \left\lbrace\int_{0}^{\tau} e^{-rs}h(s,L_{s})
ds \right\rbrace  O_{\tau}(0,0)\right].
\end{equation}
But, since for any $\nu,\xi$, the process $Y(\nu,\xi)$ is continuous, then the stopping time $\tau_{0}^{*}$ is optimal after $0$. This yields
\begin{equation}\label{equationY00}
 Y_{0}(0,0)= \E\left[\exp\left\lbrace\int_{0}^{\tau_{0}^{*}} e^{-rs}h(s,L_{s})
ds \right\rbrace O_{\tau_{0}^{*}}(0,0)\right],
\end{equation}
where
\begin{eqnarray*}
O_{\tau_{0}^{*}}(0,0) &=&\max\limits_{\beta\in U} \bigg\lbrace \E \bigg\lbrack \exp \bigg\lbrace \Int_{\tau_{0}^{*}}^{\tau_{0}^{*}+\Delta}  e^{-rs}  h(s,L_{s})ds - e^{-r(\tau_{0}^{*}+\Delta)}\psi(\beta) \bigg\rbrace Y_{\tau_{0}^{*}+\Delta}(0,\beta)|\mathcal{F}_{\tau_{0}^{*}} \bigg\rbrack  \bigg\rbrace \\
   &=&\max\limits_{\beta\in U} \bigg\lbrace \E \bigg\lbrack \exp
\bigg\lbrace \Int_{\tau_{0}^{*}}^{\tau_{0}^{*}+\Delta}  e^{-rs}
h(s,L_{s})ds - e^{-r(\tau_{0}^{*}+\Delta)}\psi(\beta) \bigg\rbrace
Y_{\tau_{0}^{*}+\Delta}(\tau_{0}^{*},\beta)|\mathcal{F}_{\tau_{0}^{*}}
\bigg\rbrack  \bigg\rbrace \\  &=& \E\bigg\lbrack \exp\bigg\lbrace
\Int_{\tau_{0}^{*}}^{\tau_{0}^{*}+\Delta}  e^{-rs}  h(s,L_{s})ds-
e^{-r(\tau_{0}^{*}+\Delta)}\psi(\beta_{0}^{*})\bigg\rbrace
Y_{\tau_{0}^{*}+\Delta}(\tau_{0}^{*},\beta_{0}^{*})|\mathcal{F}_{\tau_{0}^{*}}\bigg\rbrack,
\end{eqnarray*}
where we have used  Proposition \ref{75} iii) in the
last equality to replace $Y_{\tau_{0}^{*}+\Delta}(0,\beta)$ with
$Y_{\tau_{0}^{*}+\Delta}(\tau_{0}^{*},\beta)$. Hence,
\begin{eqnarray*}
 Y_{0}(0,0)
&=& \E\left[ \exp\bigg\lbrace \int_{0}^{\tau_{0}^{*}+\Delta} e^{-rs}h(s,L_{s}) ds -  e^{-r(\tau_{0}^{*}+\Delta)}\psi(\beta_{0}^{*})\bigg\rbrace Y_{\tau_{0}^{*}+\Delta}(\tau_{0}^{*},\beta_{0}^{*})\right].
\end{eqnarray*}
Similarly, we have 
\begin{eqnarray*}
Y_{\tau_{0}^{*}+\Delta}(\tau_{0}^{*},\beta_{0}^{*})
& =& \E\left[ \exp \bigg\lbrace \int_{\tau_{0}^{*}+\Delta}^{\tau_{1}^{*}+\Delta} e^{-rs}h(s,L_{s}+\beta_{0}^{*}) ds -  e^{-r(\tau_{1}^{*}+\Delta)}\psi(\beta_{1}^{*})\bigg\rbrace Y_{\tau_{1}^{*}+\Delta}(\tau_{1}^{*},\beta_{0}^{*}+\beta_{1}^{*})|\mathcal{F}_{\tau_{0}^{*}+\Delta}\right].
\end{eqnarray*}
Replacing this in (\ref{equationY00}), it follows that
\begin{eqnarray*}
Y_{0}(0,0) &=& \E \bigg\lbrack \exp \bigg\lbrace \int_{0}^{\tau_{0}^{*}+\Delta} e^{-rs}h(s,L_{s}) ds + \int_{\tau_{0}^{*}+\Delta}^{\tau_{1}^{*}+\Delta} e^{-rs}h(s,L_{s}+\beta_{0}^{*}) ds \\ & {}&\qquad\qquad  -   e^{-r(\tau_{0}^{*}+\Delta)}\psi(\beta_{0}^{*})-e^{-r(\tau_{1}^{*}+\Delta)}\psi(\beta_{1}^{*}) \bigg\rbrace Y_{\tau_{1}^{*}+\Delta}(\tau_{1}^{*},\beta_{0}^{*}+\beta_{1}^{*})\bigg\rbrack.
\end{eqnarray*}
Repeating this argument $n$ times, we obtain that
\begin{align}
Y_{0}(0,0) =& \E \bigg\lbrack \exp \bigg\lbrace \int_{0}^{\tau_{0}^{*}+\Delta} e^{-rs}h(s,L_{s}) ds  + \sum_{1\leq k \leq n} \int_{\tau_{k-1}^{*}+\Delta}^{\tau_{k}^{*}+\Delta} e^{-rs}h(s,L_{s}+\beta_{0}^{*}+\cdots+\beta_{k-1}^{*}) ds  \nonumber\\ &\qquad -  \sum_{k=0}^{n} e^{-r(\tau_{k}^{*}+\Delta)}\psi(\beta_{k}^{*})\bigg\rbrace Y_{\tau_{n}^{*}+\Delta}(\tau_{n}^{*},\beta_{0}^{*}+\cdots + \beta_{n}^{*})\bigg\rbrack.\lb{738}
\end{align}
But since $\P\{\tau^*_n\ge n\Delta\}=1$ then $\P$-a.s. the series
$
\sum_{n\geq 0 }e^{-r\tau_{n}^{*}}\psi (\beta_{n}^{*})
$
is convergent and \\$|\sum_{n\geq 0 }e^{-r\tau_{n}^{*}}\psi (\beta_{n}^{*})|\le C$ for some constant $C$.
On the other hand, by \eqref{estimyexp} and monotonicity, we have
$$Y^0_{\tau_{n}^{*}+\Delta}(\tau_{n}^{*},\beta_{0}^{*}+\cdots + \beta_{n}^{*})\le Y_{\tau_{n}^{*}+\Delta}(\tau_{n}^{*},\beta_{0}^{*}+\cdots + \beta_{n}^{*})\leq
\exp(\frac{\gamma e^{-r(\tau_{n}^{*}+\Delta)}}{r}).$$
As
$$\lim_{n\rightarrow \infty}Y^0_{\tau_{n}^{*}+\Delta}(\tau_{n}^{*},\beta_{0}^{*}+\cdots + \beta_{n}^{*})=1 \text{ and }\lim_{n\rightarrow \infty} \exp(\frac{\gamma e^{-r(\tau_{n}^{*}+\Delta)}}{r})=1,
$$
it follows that 
$$
\lim_{n\rightarrow \infty}Y_{\tau_{n}^{*}+\Delta}(\tau_{n}^{*},\beta_{0}^{*}+\cdots + \beta_{n}^{*})=1.
$$
Take now the limit w.r.t $n$ in the right-hand side of \eqref{738} to obtain that $Y_{0}(0,0)= J(\delta^{*})$.
\medskip

\noindent \underline{Step 2}: $J(\delta^{*})\geq J(\delta ^{\prime})$ for any other strategy  $\delta^{\prime}=(\tau^\prime_{n},\beta^\prime_{ n})_{n\geq 0} \in \mathcal{A}$. \\
\medskip

\noindent We have 
 $$ Y_{0}(0,0)\geq \E\left[ \exp\bigg\lbrace \int_{0}^{\tau_{0}^{\prime}} e^{-rs}h(s,L_{s})
ds \bigg\rbrace O_{\tau_{0}^{\prime}}(0,0)\right].
$$
Moreover, as in \eqref{inside2},
\begin{align}\nn O_{\tau_{0}^{\prime}}(0,0) =&\max\limits_{\beta \in U} \bigg\lbrace  \mathbb{E}\bigg\lbrack \exp\bigg \lbrace \int_{\tau_{0}^{\prime}}^{\tau_{0}^{\prime}+\Delta}e^{-rs}h(s,L_{s}+\xi) \1_{[s\geq \nu]}ds -e^{-r(t+\Delta)}\psi(\beta) \bigg\rbrace Y_{\tau_{0}^{\prime}+\Delta}(0,\beta)|\Fc_{\tau_{0}^{\prime}}\bigg\rbrack \bigg\rbrace.
\\\nn =&\max\limits_{\beta \in U} \bigg\lbrace  \mathbb{E}\bigg\lbrack \exp\bigg \lbrace \int_{\tau_{0}^{\prime}}^{\tau_{0}^{\prime}+\Delta}e^{-rs}h(s,L_{s}+\xi) \1_{[s\geq \nu]}ds -e^{-r(t+\Delta)}\psi(\beta) \bigg\rbrace Y_{\tau_{0}^{\prime}+\Delta}(\tau_{0}^{\prime},\beta)|\Fc_{\tau_{0}^{\prime}}\bigg\rbrack \bigg\rbrace.\\
\nn \geq & \E\left[ \exp\bigg\lbrace \Int_{\tau_{0}^{\prime}}^{\tau_{0}^{\prime}+\Delta}  e^{-rs}  h(s,L_{s})ds- e^{-r(\tau_{0}^{\prime}+\Delta)}\psi(\beta_{0}^{\prime})\bigg\rbrace  Y_{\tau_{0}^{\prime}+\Delta}(\tau_{0}^{\prime},\beta_{0}^{\prime})|\mathcal{F}_{\tau_{0}^{\prime}}\right],\end{align}
 since by Proposition \ref{75}-iii), $Y_{\tau_{0}^{\prime}+\Delta}(0,\beta)=Y_{\tau_{0}^{\prime}+\Delta}(\tau_{0}^{\prime},\beta)$ for any $\beta\in U$. Therefore, 
 \begin{eqnarray}
 Y_{0}(0,0)\geq \E\left[\exp \bigg\lbrace \int_{0}^{\tau_{0}^{\prime}+\Delta} e^{-rs}h(s,L_{s})ds - e^{-r(\tau_{0}^{\prime}+\Delta)}\psi(\beta_{0}^{\prime})\bigg\rbrace  Y_{\tau_{0}^{\prime}+\Delta}(\tau_{0}^{\prime},\beta_{0}^{\prime})\right].
\end{eqnarray}
In  a similar way,
\begin{eqnarray*}
Y_{\tau_{0}^{\prime}+\Delta}(\tau_{0}^{\prime},\beta_{0}^{\prime})& =& \esssup\limits_{\tau \in \mathcal{T}_{\tau_{0}^{\prime}+\Delta}}\E \left[\exp\bigg\lbrace
\int_{\tau_{0}^{\prime}+\Delta}^{\tau} e^{-rs}h(s,L_{s}+\beta_{0}^{\prime})ds \bigg\rbrace O_{\tau}(\tau_{0}^{\prime},\beta_{0}^{\prime})|\Fc_{\tau_{0}^{\prime}+\Delta}\right]\\
& \geq & \E \left[ \exp\bigg\lbrace \int_{\tau_{0}^{\prime}+\Delta}^{\tau_{1}^{\prime}} e^{-rs}h(s,L_{s}+\beta_{0}^{\prime})ds\bigg\rbrace  O_{\tau_{1}^{\prime}}(\tau_{0}^{\prime},\beta_{0}^{\prime})|\Fc_{\tau_{0}^{\prime}+\Delta}\right]\\
&\geq &\E \left[ \exp\bigg\lbrace \int_{\tau_{0}^{\prime}+\Delta}^{\tau_{1}^{\prime}+\Delta} e^{-rs}h(s,L_{s}+\beta_{0}^{\prime})ds-  e^{-r(\tau_{1}^{\prime}+\Delta)}\psi(\beta_{1}^{\prime})\bigg\rbrace  Y_{\tau_{1}^{\prime}+\Delta}(\tau_{1}^{\prime},\beta_{0}^{\prime}+\beta_{1}^{\prime})|\Fc_{\tau_{0}^{\prime}+\Delta}\right].
\end{eqnarray*}
Therefore,
\begin{eqnarray*}
Y_{0}(0,0) & \geq & \E \bigg\lbrack \exp \bigg\lbrace \int_{0}^{\tau_{0}^{\prime}+\Delta} e^{-rs}h(s,L_{s})ds+ \int_{\tau_{0}^{\prime}+\Delta}^{\tau_{1}^{\prime}+\Delta} e^{-rs}h(s,L_{s}+\beta_{0}^{\prime})ds \\ &-&  e^{-r(\tau_{0}^{\prime}+\Delta)}\psi(\beta_{0}^{\prime})- e^{-r(\tau_{1}^{\prime}+\Delta)}\psi(\beta_{1}^{\prime})\bigg\rbrace  Y_{\tau_{1}^{\prime}+\Delta}(\tau_{1}^{\prime},\beta_{0}^{\prime}+\beta_{1}^{\prime})\bigg\rbrack.
\end{eqnarray*}
 Repeat this argument $n$ times to obtain
\begin{eqnarray*}
Y_{0}(0,0) & \geq & \E \bigg\lbrack \exp\bigg\lbrace  \int_{0}^{\tau_{0}^{\prime}+\Delta} e^{-rs}h(s,L_{s})ds+\sum_{1\leq k \leq n}\int_{\tau_{k-1}^{\prime}+\Delta}^{\tau_{k}^{\prime}+\Delta} e^{-rs}h(s,L_{s}+\beta_{0}^{\prime}+\cdots + \beta_{k-1}^{\prime})ds \\ &-& \sum_{k=0}^{n}e^{-r(\tau_{k}^{\prime}+\Delta)}\psi(\beta_{k}^{\prime}) \bigg\rbrace Y_{\tau_{n}^{\prime}+\Delta}(\tau_{n}^{\prime},\beta_{0}^{\prime}+\cdots +\beta_{n}^{\prime})\bigg\rbrack.
\end{eqnarray*}
Now, we take the limit as $n\to +\infty$ in the right hand-side of this inequality to obtain that
$$ Y_{0}(0,0) \geq  \E \left[ \exp \bigg\lbrace \int_{0}^{+\infty } e^{-rs}h(s,L_{s}^{\delta^{\prime}})ds - \sum_{n\geq 0}e^{-r(\tau_{n}^{\prime}+\Delta)}\psi(\beta_{n}^{\prime}) \bigg\rbrace \right]= J(\delta^{\prime})$$since the series is convergent and bounded and, as above, $\underset{n\to\infty}{\lim}Y_{\tau_{n}^{\prime}+\Delta}(\tau_{n}^{\prime},\beta_{0}^{\prime}+\cdots +\beta_{n}^{\prime})=1$. This latter point can be obtained by \eqref {estimyexp} and the fact that
$Y(\nu,\xi)\ge Y^0(\nu,\xi)$. Therefore, $ Y_{0}(0,0) \geq   J(\delta^{\prime})$. Thus, we conclude that for any arbitrary strategy $\delta$ in $\mathcal{A}$, we have that
 $$ Y_{0}(0,0)=J(\delta^{*})=\Sup_{\delta \in \mathcal{A}}J(\delta)$$which means that $\delta^{*}$ is optimal.
\end{proof}
\section{Appendix}
\noindent \underline{\bf Part (I)}: \textbf{Snell envelope.}
\medskip

\noindent Let $U$ be an $\mathcal{F}$-adapted c\`adl\`ag process which belongs to class $[D]$, i.e. the random variables set $\{U_\theta, \theta \in \mathcal{T}\}$ is uniformly integrable. The Snell envelope of the process $U$ denoted by $SN(U)$ is the smallest c\`adl\`ag super-martingale which dominates $U$. It exists and satisfies
\begin{enumerate}[i)]
\item
\begin{equation}
\forall t\geq 0,\quad SN_{t}(U):=\esssup\limits_{\theta \in \mathcal{T}_{t}}\E[U_{\theta}|\Fc_{t}].
\end{equation}
\item $\lim_{t\rightarrow \infty}SN_t(U)=\limsup_{t\rightarrow \infty}U_t$.
\item The jumping times of $(SN_{t}(U))_{t\geq 0}$ are predictable and verify
$\{\Delta (SN_{t}(U))<0\}\subset \{SN_{t^{-}}(U)=U_{t^{-}}\}\cap \{\Delta_{t} U<0\}$.
\item  If $U$ has only positive jumps on $[0,\infty]$, then $SN(U))$ is a continuous process on $[0,\infty]$. Moreover, if $\theta$ is an $\Fc_{t}$-stopping time and, $\tau^{*}_{\theta}= \inf\{ s\geq \theta, SN(U)_{s}\leq U_{s}\}$
($+\infty$ if empty), then $\tau^{*}_{\theta}$ is optimal after $\theta$, i.e.,
\begin{equation}
SN(U)_{\theta}=\mathbb{E}[SN(U)_{\tau^{*}_{\theta}}|\Fc_{\theta}]= \mathbb{E}[U_{\tau^{*}_{\theta}}|\Fc_{\theta}]= \esssup\limits_{\tau \geq \theta} \mathbb{E}[U_{\tau}|\Fc_{\theta}].
\end{equation}
\item If $(U_n)_{n\geq 0}$ and $U$ are c\`adl\`ag processes of class $[D]$ and such that the sequence of process
$(U_n)_{n\geq 0}$ converges increasingly and pointwisely to $U$, then $ (SN(U_{n}))_{n\geq 0}$
converges increasingly and pointwisely to $SN(U)$.
\medskip

For further reference and details on the Snell envelope, we refer to \cite{El-Karoui1} or \cite{Dellacherie}.

\end{enumerate}
\noindent \underline{\bf Part (II)}: \textbf{Optional and predictable projections}
\medskip

\noindent Let $X:=(X_t)_{t\geq 0}$ be a measurable bounded process.
\medskip

\noindent i) There exists an optional (resp. predictable) process $Y$ (resp. $Z$) such that
$$\E[X_T\1_{\{T<\infty\}}|\Fc_T]=Y_T\1_{\{T<\infty\}} , \,\,\P-a.s. \mbox{ for any stopping time }T$$
(resp.
$$\E[X_T\1_{\{T<\infty\}}|\Fc_{T^{-}}]=Z_T\1_{\{T<\infty\}} , \,\,\P\mbox{-}a.s. \mbox{ for any predictable stopping time }T).$$
The process $Y$ (resp. $Z$) is called the optional (resp. predictable) projection of the process $X$.\\
\medskip
\noindent ii) If $X$ is c\`adl\`ag, then $Y$ is also c\`adl\`ag.
\medskip

\noindent iii) Since the filtration $(\Fc_t)_{t\geq 0}$ is Brownian then $\Fc_{T^{-}}=\Fc_{T}$ and the processes $Y$ and $Z$ are undistinguishable. In particular, the optional projection of a bounded continuous process is also continuous. Finally for any predictable stopping time $T$
$$
\E[\Delta _TX|\Fc_{T}]=\Delta_T Z,\quad \P\mbox{-}a.s.
$$For more details one can see (\cite{Dellacherie}, pp.113,  ).

\end{document}